\documentclass[12pt]{amsart}

\usepackage[margin=1in]{geometry}  
\usepackage{graphicx}              
\usepackage{amsmath}               
\usepackage{amsfonts}              
\usepackage{amsthm}                
\usepackage{amssymb}
\usepackage{enumerate}
\usepackage{graphicx}
\usepackage{tabularx}
\graphicspath{ {./images/} }
\usepackage{xargs}
\usepackage{geometry}

\usepackage{fancyhdr} 
\usepackage{lastpage} 
\usepackage{extramarks} 
\usepackage[most]{tcolorbox} 
\usepackage{xcolor} 
\usepackage[hidelinks]{hyperref} 
\usepackage[permil]{overpic}
\usepackage{pict2e} 
\usepackage{listings}

\newtheorem{Conjecture}{Conjecture}
\newtheorem{Theorem}{Theorem}

\newtheorem{Proposition}{Proposition}
\newtheorem{Corollary}{Corollary}

\newtheorem{Definition}{Definition}

\usepackage[T2A,T1]{fontenc}
\newcommand{\Zhe}{\mbox{\usefont{T2A}{\rmdefault}{m}{n}\CYRZH}}

\begin{document}

\title{ Period-Doubling Cascades Invariants:\\ Braided Routes To Chaos }

\author{Eran Igra and Valerii Sopin}

  
\email{vvs@myself.com, vsopin@simis.cn}
\email{eranigra@simis.cn}
\maketitle

\begin{abstract} 
By a classical result of Kathleen Alligood and James Yorke we know that as we isotopically deform a map $f:ABCD\to\mathbb{R}^2$ to a Smale horseshoe map we should often expect the dynamical complexity to increase via a period--doubling route to chaos. Inspired by this fact and by how braids force the existence of complex dynamics, in this paper we introduce three topological invariants that describe the topology of period--doubling routes to chaos. As an application, we use our methods to ascribe symbolic dynamics to perturbations of the Shilnikov homoclinic scenario and to study the dynamics of the Henon map.
\end{abstract}

\section{Introduction}
Consider a continuous one-parameter family of homeomorphisms $H_t:\mathbb{R}^2\to\mathbb{R}^2$, $t\in[0,1]$ s.t. the following conditions are satisfied:
\begin{itemize}
    \item For all $t$, there is a bounded attractor $A$, which attracts a.e. initial condition in $\mathbb{R}^2$.
    \item For $H_0$, the attractor $A$ includes at most a finite number of periodic orbits.
    \item For $H_1$, the dynamics on $A$ includes a Smale horseshoe map, hence they are chaotic (see $[36]$ for the definition).
\end{itemize}
In this paper, we are interested in the following question - how can we describe the evolution of the dynamics along the isotopy from  "simple" dynamics into chaotic, horseshoe dynamics?\\

As proven in $[27]$, in many cases the complex dynamics in $A$ would be the result of a period-doubling route to chaos. In fact, as shown in $[28, 69]$, in many cases one should expect the complex dynamics of $A$ to appear via a period-doubling route to chaos. The results of $[27,28,69]$ are independent of the dimension, in the sense that they depend on certain approximation arguments which hold in $\mathbb{R}^n$ - for all $n\geq2$. In the special case when $n=2$ one could also tackle this problem using tools from two-dimensional topology -- namely, the theory of surface dynamics (see $[33]$ for a survey). In particular, one could study this question using braids - which can be thought of as "mixing patterns" for two-dimensional homeomorphisms. This approach was adopted in $[47, 54, 57]$ (among others), where it was applied to study the evolution of the Henon map (see $[53]$) from order into chaos. In this paper we bridge the approaches of $[27,28,69]$ and $[47,54,57]$. In detail, we use both these approaches study period-doubling routes to chaos based on their topology.\\

To make this statement precise (and using previous notations), from now on denote by $Per=\{(x,t)\in\mathbb{R}^2\times[0,1]|\; \exists n\geq1, H^n_t(x)=x\}$ the collection of periodic points for the isotopy $H_t:\mathbb{R}^2\to\mathbb{R}^2$, $t\in[0,1]$. Generically, the set $Per$ is a collection of curves which are branched at the bifurcation orbits - let $P\subseteq Per$ denote the collection of period-doubling cascades. By the type of the cascade, we will broadly mean the topological type of the set $S=(\mathbb{R}^2\times[0,1])\setminus P$ (for the precise details, see Definition \ref{type}).\\

The topology of $S$ could be difficult to analyze directly, as $P$ could be visualized as an infinite collection of "branched braids" (see the illustration in Figure \ref{branched}). To bypass this obstacle, we use the fact that given any $t\in[0,1]$, every finite collection of periodic orbits in $P\cap(\mathbb{R}^2\times\{t\})$ forms a braid. This allows us to study the topology of $S$ by studying its "slices" using the Betsvina-Handel algorithm. Motivated by the rich theory of braids and its connections with various fields in mathematics, we reduce the topology of $S$ to three major invariants as described below.\\ 

The first invariant, the \textbf{Conformal Index}, is derived from the connection between braids and conformal structures on surface s as outlined in $[29]$ (see Definition \ref{chaoticdil}, Theorem \ref{ch} and Theorem \ref{PA}). Despite its seemingly abstract nature, in Subsection \ref{henap} and Subsection \ref{shill} we illustrate how these methods can be applied to study different cases where complex dynamics arise. In particular, in Subsection \ref{shill} we apply this invariant to ascribe symbolic dynamics to perturbations of Shilnikov's homoclinic scenario (see Theorem \ref{shilPA}).\\

Following that, we proceed to analyze the topology of $S$ using both their group-theoretic properties and representation theory of braid groups (in particular, the Burau representation). This allows us to define two invariants, the \textbf{Index-Invariant} and the \textbf{Trace-Invariant}, which describe the topology of $S$ (see Theorem \ref{numinv}). Moreover, unlike the Conformal Index, these two invariants can possibly be approximated numerically (see Subsection \ref{numerin2}).\\

This paper is organized as follows. In Section $2$ we introduce several prerequisites from braid theory, the theory of surface dynamics, and the theory of quasiconformal mappings. Following that, in Section $3$ we rigorously define the notion of a period-doubling route to chaos used in this paper and prove our results. Finally, we state that our results appear to suggest deep connections between period-doubling routes to chaos and several fields of mathematics - including knot theory, higher Teichmüller theory and number theory. Therefore, in Section $4$ (and at the end of Subsection \ref{numerin}) we discuss the possible generalizations and extension of our work, and conjecture how it possibly forms a part of some larger theory.\\

Before we begin, we make a final remark that much of our results were motivated by an attempt to describe the evolution of dynamics from order into chaos using algebraic methods. As such, even though it is not at all clear from the text, our ideas were, to a large degree, inspired by the Conley Index theory (see $[86,87]$). In fact, our results originated by an attempt to study the Braid Type Conjecture $[54]$ using homology and representation-theoretic tools.

\subsection*{Acknowledgements}
The authors are grateful to Weixiao Shen, Assaf Bar Natan, Jorge Olivares Vinales, and Joshua Haim Mamou for their helpful suggestions and support.

\section{Prerequisites}

In this section, we survey the major tools we use throughout this paper. We begin with Subsection \ref{braidef}, where we survey the basics of braid theory. Following that, in Subsection \ref{indexdef} we look over several facts from representation theory, after which we review some number--theoretic results in Subsection \ref{numdef}. Finally, in Subsection \ref{topod} we review the connection between braids, dynamics, and complex structures.

\subsection{Braid groups}
\label{braidef}

In this subsection, we review some basic facts about the braid groups.

\subsubsection{Braids}

Consider a collection of $\infty\geq n\geq2$ disjoint curves in $\mathbb{R}^2 \times [0, 1]$, to which we sometimes refer as \textbf{strands}, $S_1,...,S_n$ s.t. the following is satisfied:

\begin{itemize}
    \item Given any $t\in[0,1]$, $S$ intersects $\mathbb{R}^2\times\{t\}$ in precisely one point.
    \item   The strands are mutually disjoint (see the illustration in Figure \ref{horseshoebraid}).
    \item For each $1\leq i\leq n$, the points $\alpha_i=S_i\cap(\mathbb{R}^2\times\{0\})$ and $\omega_i=S_i\cap(\mathbb{R}^2\times\{1\})$ lie inside $[0,1]\times\{0\}$ and $[0,1]\times\{1\}$ (respectively).
\end{itemize}

\begin{figure}[h]
\centering
\begin{overpic}[width=0.25\textwidth]{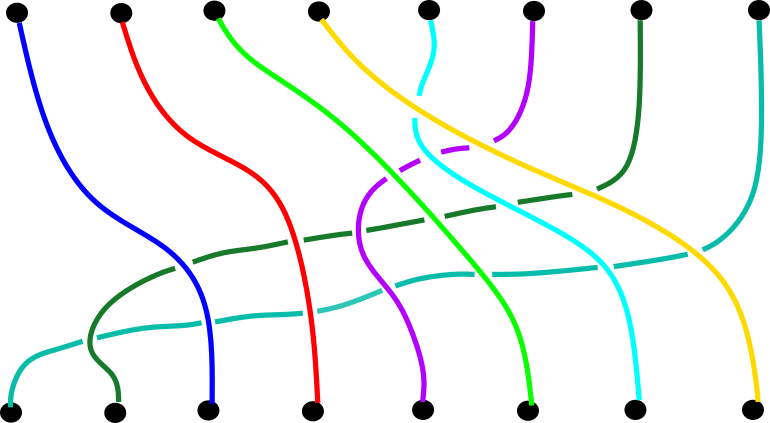}
\end{overpic}
\caption{\textit{A braid corresponding to a periodic orbit of minimal period 8 (a braid on 8 strands) for the $\cap$-Smale horseshoe (see Figure 42 in $[54]$).}}\label{horseshoebraid}

\end{figure}

We define the \textbf{braid generated by $S_1,...,S_n$} (or just a \textbf{braid} when $S_1,...,S_n$ are obvious from context) as the topological type of $\mathbb{R}^2\times(0,1)\setminus(S_1\cup...\cup S_n)$. In other words, the space $\mathbb{R}^3\times[0,1]\setminus(S_1\cup...\cup S_n)$ can be isotopically deformed as long as two conditions are satisfied:
\begin{enumerate}
    \item The respective locations of  $\alpha_i$ and $\omega_i$ inside $[0,1]\times\{0\}$ and $[0,1]\times\{1\}$ remain fixed.
    \item The topological type of $\mathbb{R}^2\times(0,1)\setminus(S_1\cup...\cup S_n)$ also remains fixed throughout the isotopy.
\end{enumerate}

\begin{figure}[h]
\centering
\begin{overpic}[width=0.45\textwidth]{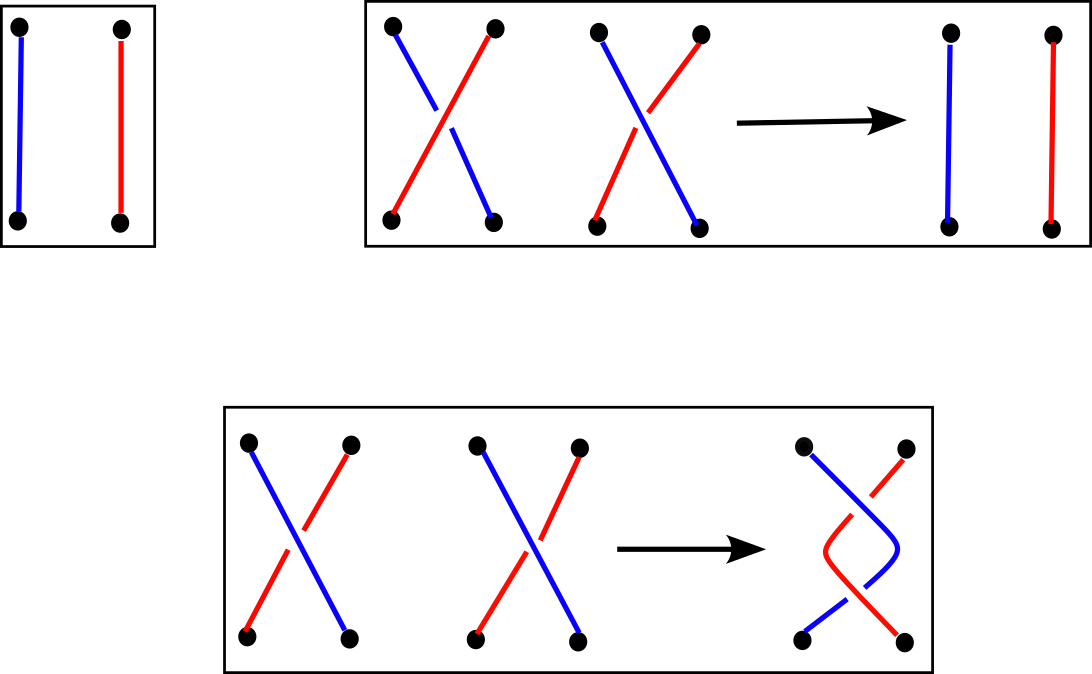}
\end{overpic}
\caption{\textit{Braids in $B_2$ - on the upper left side we have the identity braid, on the upper right side we have a braid and its inverse, and on the bottom we have an example of a non-trivial braid generated by concatenating two braids.}}\label{conc}

\end{figure}

Braids can also be easily described using group theory. To do so, given any $n\geq2$ recall that we can concatenate braids to create a new braid, define inverses and define an identity braid (see the illustration in Figure \ref{conc}). This shows that we can speak of $B_n$ as the \textbf{braid group on $n$ strands}. In light of the above, $B_n$ can be viewed as the fundamental group of the space of $n$ unordered, distinct points in the plane takes place. A path in this space then permutes the points, giving a surjection of $B_n$ onto the symmetric group, $S_n$. The kernel of this map is a subgroup of $B_n$ known as the pure braid group and denoted $P_n$. Also connected with the above is the definition of braids as certain automorphisms of the free group, $F_n$, which is the fundamental group of the complement of the space of $n$ points in $\mathbb{R}^2$.\\


In this paper, we will study braids as Mapping Class Groups. In more detail, we will consider the braid group $B_n$ as the collection of all isotopy classes of all orientation-preserving homeomorphisms of the \textbf{n-punctured disc}. It is easy to see that this object can be induced with a natural group structure, hence the name \textbf{Mapping Class Group} (see $[1]$ and the beginning of Subsection \ref{topod}).

\subsubsection{Left-invariant linear ordering}

From now on, we denote by $B_{\infty}$ for the group generated by an infinite sequence of the standard generators $\sigma_i$, i.e., the direct limit of all $B_n$ with respect to the inclusion of $B_{n}$ into $B_{n+1}$ (for more details, see $[2]$).\\ 

The braid groups are equipped with a distinguished strict \textbf{linear ordering}, which is compatible with multiplication on the left, and admits a simple characterization: a braid $x$ is smaller than another braid $y$ if, among all expressions of the quotient $x^{-1}y$ in terms of the generators $\sigma_i, i=1,...,n$, there exists at least one expression in which the generator $\sigma_m$ with maximal (or minimal) index $m$ appears only positively, that is, $\sigma_m$ occurs, but $\sigma^{-1}_m$ does not. In addition, the ordered set $B_{\infty}$ is order-isomorphic to ($\mathbb{Q}$, $<$) (see $[2]$ for the details).

\subsubsection{Linear representations}

An open question for a while was the question of whether braid groups are linear, that is isomorphic to a linear group (a matrix group). Several linear representations of $B_n$ have been considered, the first being the \textbf{Burau representation} $[3]$ of $B_n$ in $GL_{n}(\mathbb{Z}[t^{\pm  1}])$, which turned out not to be faithful (injective) for large $n$ (but it is faithful for $n=3$), see $[4]$. The Grassner representation $[5]$ is a natural variant of the Burau representation, but it is only defined on pure braids (see, for example, $[6, 7]$ where its faithfulness is discussed). The first linear representation proven to be faithful was the Lawrence-Krammer representation of $B_n$ in $GL_{n(n-1)/2}(\mathbb{Z}[t^{\pm 1}, q^{\pm 1}])$, a deformation of the symmetric square of the
Burau representation $[8, 9]$.\\

The Burau, Grassner and Lawrence-Krammer representations are known to be unitary with respect to some Hermitian form $[6, 10, 11]$. The image of these representations is dense in the
unitary group $[12]$.

\subsection{Index}
\label{indexdef}
In this subsection, we closely consider the Burau representation and the symplectic representation coming from it. We also recall the concept of index of a subgroup and the strong approximation theorem.

\subsubsection{The Burau representation}
\label{bur}

Braids also come into knot theory, since they may be closed to form links by connecting the bottoms of the strands to the tops. The Burau representation of a braid is the action on the first homology of the punctured disk. The Alexander polynomial of a knot is an invariant of the first homology of the knot complement. In fact, the determinant of 1-Burau(a braid) is a multiple of the Alexander polynomial of that braid $[8]$.\\

Let $\sigma_i$ denote the standard generators of the braid group $B_n$. Then for any non-zero complex $t$, the \textbf{Burau representation} can be described explicitly by mapping (where $I_k$ denotes the $k \times k$ identity matrix) $[3, 4]$:

$$\sigma_i \longrightarrow \begin{pmatrix}
    I_{i-1} & 0 & 0 &  0 \\
    0 & 1-t & t  & 0 \\
    0 & 1 & 0 & 0 \\
    0 & 0 & 0 & I_{n-i-1}
\end{pmatrix}$$

In $[13]$ it was shown that the kernel of the representation is infinitely generated for $n \geq 6$. 
\begin{figure}[h]
\centering
\begin{overpic}[width=0.3\textwidth]{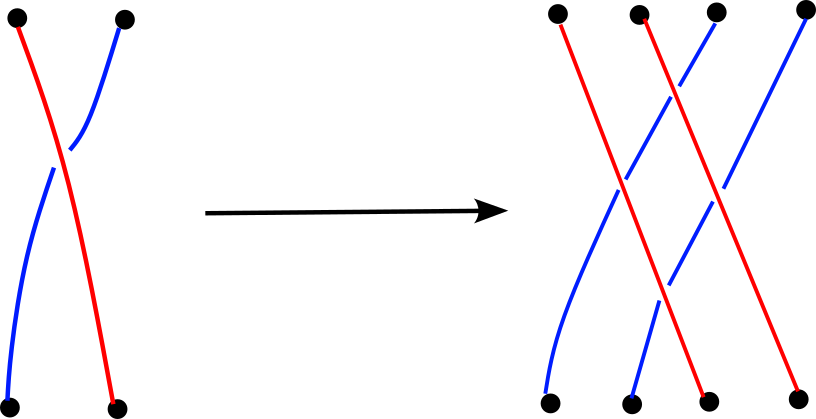}
\end{overpic}
\caption{\textit{The cabling operation.}}\label{cab}
\end{figure}

Moreover, the Burau representation enables us to define many other representations of the braid group $B_n$ by the topological operation of “cabling braids”. In this operation, each strand is replaced by $r$ parallel strands, as illustrated in Figure \ref{cab}. These representations split into copies of the Burau representation itself and of a representation of $B_n/(P_n, P_n)$. In particular, there is no gain in terms of faithfulness by cabling the Burau representation (for the precise details, see $[14]$).

\subsubsection{The Special linear group}
\label{sl}
The special linear group $SL(n, R)$ of degree $n$ over a commutative ring $R$ is the set of $n \times n$ matrices with determinant $1$, with the group operations of ordinary matrix multiplication and matrix inversion. This is a normal subgroup of the general linear group $GL(n, R)$ given by the kernel of the determinant (any invertible matrix can be uniquely represented according to the polar decomposition as the product of a unitary matrix and a hermitian matrix with positive eigenvalues).\\

By taking $t=-1$ in the Burau representation of $B_n$, the resulting target group is $SL(n, \mathbb{Z})$ - a \textbf{symplectic representation} of the braid group $[15]$. We will also need the following result, called the \textbf{Strong Approximation theorem}: 

\begin{Theorem}
    \label{strongapp} Given any $N\geq2$, let denote $\mathbb{Z}_N$ the ring of integers modulo $N \geq 2$. Then, the natural map $SL(n, \mathbb{Z}) \longrightarrow SL(n, \mathbb{Z}_N)$ defined by replacing each entry $a_{i,j}$ of $A\in SL(n,\mathbb{Z})$ with ${a_{i,j}} \mod N$ is onto.
\end{Theorem}

In addition, all the algebraic relations between characters of the Special Linear group can be visualized as relations between graphs (resembling \textbf{Feynman diagrams}), see $[16]$ for the details. Moreover, a parametrization of the irreducible complex characters of the finite special linear groups is given in $[17]$.

\subsubsection{Index of subgroups of the special linear group}

The index of a subgroup $H$ in a group $G$ is the number of left cosets of $H$ in $G$, or equivalently, the number of right cosets of $H$ in $G$. Because $G$ is the disjoint union of the left cosets and because each left coset has the same size as $H$, the index measures the “relative sizes”.\\

It is a classical result that $SL(n, \mathbb{Z})$ is finitely generated, and hence so are all its subgroups of finite index $[18]$. The level $N$ principal congruence subgroup of $SL(n, \mathbb{Z})$ is the kernel of the homomorphism $SL(n, \mathbb{Z}) \longrightarrow SL(n, \mathbb{Z}_N)$ that reduces the entries in matrices modulo $N$. Clearly $SL(n, \mathbb{Z}_N)$ is finite-index in $SL(n, \mathbb{Z})$. A subgroup $G$ of $SL(n, \mathbb{Z})$ is a \textbf{congruence subgroup} if there exists some $N \geq 2$ such that $SL(n, \mathbb{Z}_N)$ is a subset of $G$. The level of $G$ is then the smallest such $N$. From this definition it follows that congruence subgroups are of finite index and the congruence subgroups of level $N$ are in one-to-one correspondence with the subgroups of $SL(n, \mathbb{Z}_N)$. Therefore, to find the index it is sufficient to find the level.\\

For $n \geq 3$, every \textbf{finite-index subgroup} of $SL(n, \mathbb{Z})$ is a congruence subgroup $[19]$. Non-congruence subgroups of $SL(2, \mathbb{Z})$ exist, since $SL(2, \mathbb{Z})$ contains a free subgroup of finite index, and thus contains a veritable zoo of finite-index subgroups -- most of these are not congruence subgroups. Moreover, maximal subgroups of infinite index in $SL(n, \mathbb{Z})$ exist and there are uncountably many such subgroups $[20]$. In addition, $SL(n, \mathbb{Z})$ contains free subgroups. Therefore, it is impossible to have a decision algorithm that can always answer whether elements $SL(n,\mathbb{Z})$ generate a subgroup of finite index. And yet, by enumerating cosets, if the index is finite, the process will terminate, and give the correct index (the \textbf{Todd–Coxeter algorithm $[88]$}).

\subsection{Links to Number Theory}
\label{numdef}
In this subsection, we give a list of ways in which number theory deals with sequences of natural numbers. We will use these ideas in Subsection \ref{numerin}.

\subsubsection{Continued fractions}

A continued fraction is a mathematical expression that can be written as a fraction with a denominator that is a sum that contains another simple or continued fraction. Depending on whether this iteration terminates with a simple fraction or not, the continued fraction is finite or infinite. The continued fraction representation for a real number is finite if and only if it is a rational number.\\

In detail, given a sequence $\{a_{i}\}$ of integers, the corresponding \textbf{continued fraction} (sometimes called "simple continued fraction") is an expression of the form: 
$$a_1+\cfrac{1}{a_2+\cfrac{1}{a_3+\cfrac{1}{a_4+\cfrac{1}{a_5+
  \cdots\vphantom{\cfrac{1}{1}} }}}}$$

If every $a_i$ is a natural number (a regular continued fraction), then the (infinite) continued fraction converges to the irrational limit (the simple continued fraction representation of an irrational number is unique) $[22]$.

\subsubsection{P-adics}

Let $p$ be an arbitrary prime. The p-adic valuation is denoted by
$|*|_p$. This valuation satisfies the \textbf{strong triangle inequality}: $|x + y|_p \leq \max(|x|_p, |y|_p).$ It is the main distinguishing property of the p-adic valuation inducing essential departure from real or complex analysis (see $[23, 24, 25]$ for the details).\\

Any p-adic integer, represented as a sequence of residues, can be expanded into the series $$\sum\limits_{i=0}^\infty\alpha_i p^i,$$ $\text{where } \alpha_i\in\{0,\;\dots,\;p-1\}, i\in\mathbb{N}$, see $[23, 24, 25]$. As shown in $[26]$, the p-adic integers can be visualized as a subset of the plane using a fractal-like construction. Moreover, one can define the n-adic integers, even if $n$ is not prime $[23]$. Finally, we recall that a p-adic integer is a rational number precisely when the sequence of its digits is eventually periodic $[23, 24, 25]$.

\subsection{Forcing theory and the Betsvina-Handel algorithm}
\label{topod}
In this subsection, we review the connection between braids and dynamics, following the ideas introduced in $[33, 34, 35]$. To begin, let $\mathbb{D}$ denote an open disc, and let $f:\mathbb{D}\to \mathbb{D}$ denote an orientation-preserving homeomorphism. Let $\gamma$ denote a finite set with at least two points, invariant under $f$ - specifically, let $\gamma$ denote a finite collection of periodic orbits for $f$ in $\gamma$. We will often refer to $\gamma$ as \textbf{the punctures}, as in this section we study the isotopy class of $f$ in $\mathbb{D}\setminus\gamma$. Moreover, as will be clear from the statements below, the disc $\mathbb{D}$ could easily be replaced with $\mathbb{R}^2$ or the sphere $S^2$.\\

We begin by recalling that given any $\gamma$ as above, there exists a graph $\Gamma\subseteq\mathbb{D}$, whose vertices are the points of $\gamma$, which is a retract of the surface $\mathbb{D}\setminus\gamma$ (see the illustration in Figure \ref{russiang}). In particular, if $\gamma$ has $n$ points, the graph $\Gamma$ has $n-1$ edges and the action of $f$ on these edges defines a braid, as illustrated in Figure \ref{russiang}. Alternatively, one could also visualize this construction by stretching $\mathbb{D}$ into the tube $\mathbb{D}\times[0,1]$ where each point $(x,0)$ is connected by an arc to $(f(x),1)$. It is easy to see the arcs connecting $\gamma\times\{0\}$ and $\gamma\times\{1\}$ then form a braid.\\

\begin{figure}[h]
\centering
\begin{overpic}[width=0.4\textwidth]{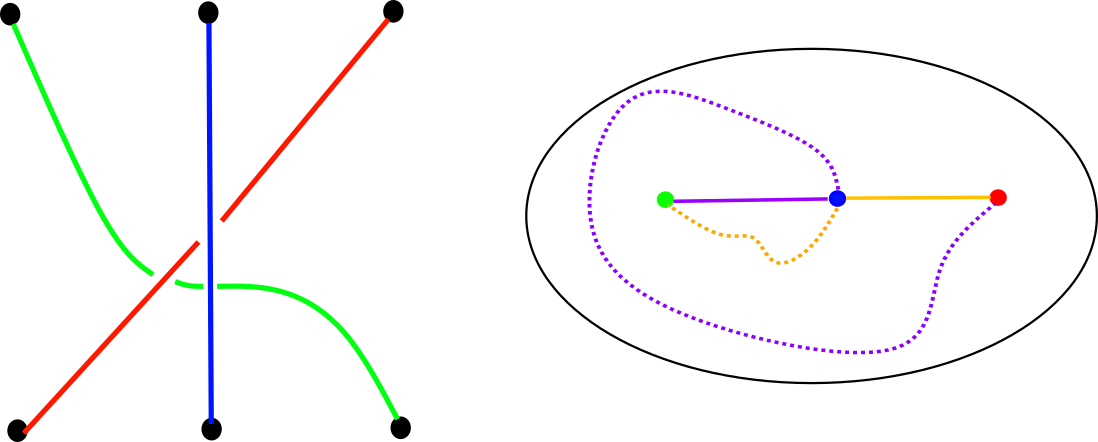}
\end{overpic}
\caption{\textit{On the right - a disc diffeomorphism that permutes the red and green points and fixes the blue points, on the left - the action on the purple and yellow arcs (i.e., the dashed lines) forces the creation of a braid. We denote this braid by \textbf{Russian G:}  $\Zhe$.}}\label{russiang}

\end{figure}
Dynamically, one should think of the correspondence between periodic orbits and braids as "folding patterns" or "twisting patterns" which describe how a given homeomorphism $f:\mathbb{D}\to\mathbb{D}$ mixes the initial conditions in the disc. As such, inspired by the Sharkovskii theorem $[38]$ and the Li--Yorke theorem $[39]$ from one-dimensional dynamics, we now ask the following question: given a homeomorphism $f:\mathbb{D}\to\mathbb{D}$, when would it generate a periodic orbit of some prescribed braid type, $\rho$? To state the answer to this question, we first recall the following definition:

\begin{Definition}
    \label{force} We say that a given braid $\gamma$ \textbf{forces} \textbf{the existence of another braid} $\rho$ if any orientation-preserving homeomorphism $f:\mathbb{D}\to \mathbb{D}$ which generates $\gamma$ as a collection of periodic orbits also generates a collection of periodic orbits corresponding to $\rho$. We denote it by $\gamma\geq\rho$ - it is easy to see this defines a partial order on the collection of braids (see $[33]$).
\end{Definition}
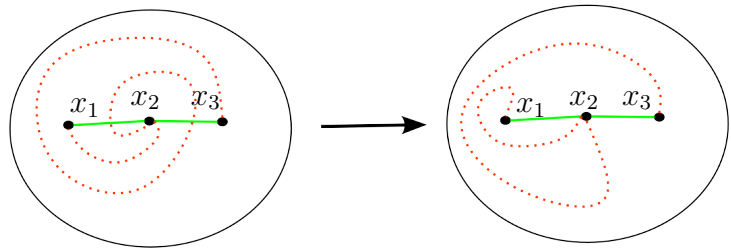
\begin{figure}[h]
\centering
\begin{overpic}[width=0.6\textwidth]{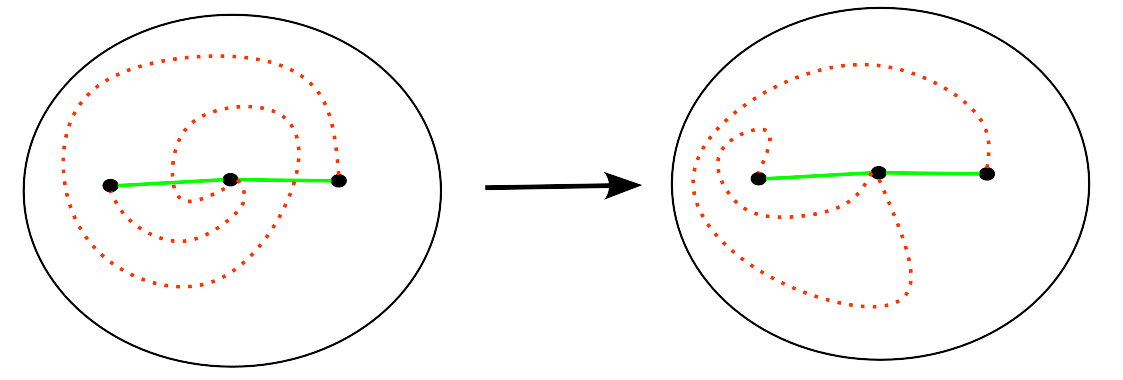}
\put(690,190){$x_1$}
\put(830,200){$x_3$}
\put(760,200){$x_2$}
\put(100,190){$x_1$}
\put(180,200){$x_{2}$}
\put(260,200){$x_3$}
\end{overpic}
\caption{\textit{The homeomorphisms $h$ (on the left) and $g$ (on the right) are isotopic on $D$ $rel$ $\gamma$, where $\gamma=\{x_1,x_2,x_3\}$ (in this scenario, $h(x_1)=g(x_1)=x_2$, $h(x_2)=g(x_2)=x_1$ and $x_3$ remains fixed). This is exemplified by how they distort the green curves connecting the elements in $\gamma$ (i.e., the spine of $S$) - whose respective images under $h$ and $g$ are the dashed red lines.}}\label{relative}

\end{figure}

To state the answer using this term, we will need another definition:

\begin{Definition}
    \label{relp}
With previous notations, we say an orientation-preserving homeomorphism $g:\mathbb{D}\to \mathbb{D}$ is \textbf{isotopic to $f$ relative to} $\gamma$ (or in short, $rel$ $\gamma$) if there exists an isotopy $f_t:D\to D$, $t\in[0,1]$ of continuous maps satisfying (see the illustration in Figure \ref{relative}):
    \begin{itemize}
        \item $f_0=f$, $f_1=g$.
        \item For every $s,t\in[0,1]$ $f_t$ and $f_s$ permute the points in $\gamma$ in the exact same way, and define the same braid.
    \end{itemize}
\end{Definition}

As proven in $[35]$, the answer to this question is given via the Thurston-Nielsen Classification theorem $[32]$. To state that theorem, we need to introduce the notion of a pseudo-Anosov map:

\begin{Definition}
    \label{pseanosov} Let $S$ be a surface with punctures $x_1,...,x_n$, $n\geq0$. An orientation-preserving homeomorphism $h:S\to S$ is \textbf{pseudo-Anosov} provided there exist two foliations of $S$, $F^u$ and $F^s$, transverse to one another throughout $S$ (but not necessarily at the punctures $\{x_1,...,x_n\}$ - see the illustration in Figure \ref{trannn}) and some $\lambda>1$ s.t. the following is satisfied:
    \begin{itemize}
        \item Both $F^s$ and $F^u$ are measured - i.e., if we move some leaf $L_1$ of $F^i$ to another leaf $L_2$ of $F^i$ by some isotopy of $S$ (where $i\in\{u,s\}$), the Borel measure on $L_2$ is the pushforward of the Borel measure on $L_1$.
        \item $h(F^u)=\lambda F^u$ while $h(F^s)=\frac{1}{\lambda}F^s$ - i.e., $h$ stretches uniformly the unstable foliation $F^u$ and squeezes uniformly the stable foliation $F^s$. 
    \end{itemize}
\end{Definition}

\begin{figure}[h]
\centering
\begin{overpic}[width=0.5\textwidth]{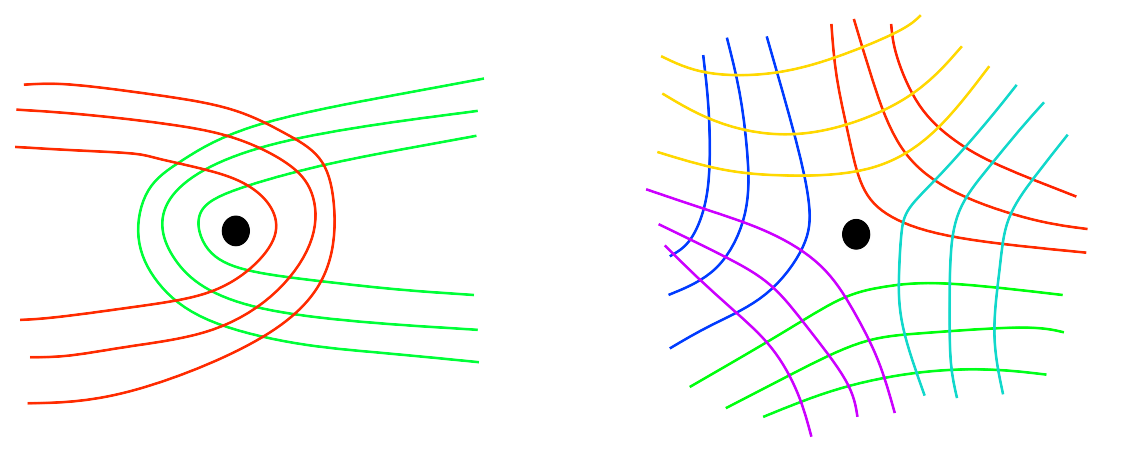}
\end{overpic}
\caption[Transverse folliations.]{\textit{Transverse foliations (with singularities) around the punctures of $S$. }}
\label{trannn}
\end{figure}

Pseudo-Anosov maps are essentially a generalized form of hyperbolic diffeomorphisms like Smale's horseshoe $[36]$ - that is, they contract uniformly in one direction and expand uniformly in another. The reason we are interested in pseudo-Anosov maps is because as far as their dynamics are concerned, we have the following result (see Theorem $7.2$ in $[33]$, Theorem $2$ and Remark $2$ in $[34]$, and along Subsections $3.4$ and $4.4$ in $[35]$):

\begin{Theorem}
\label{stability}    Assume $S$ is a surface of negative Euler characteristic, and that $h:S\to S$ is pseudo-Anosov, then $h$ has infinitely many periodic orbits, satisfying the following:
\begin{itemize}
    \item If $x$ is periodic of minimal period $k$ for $h$, then as $g$ is isotoped to some other homeomorphism $f:S\to S$ the point $x$ is continuously deformed to $y$ - a periodic point for $f$ of minimal period $k$. That is, the periodic dynamics of $h$ are \textbf{unremovable}.
    \item  When we isotope $g$ to $f$, no two periodic orbits of $h$ collapse into one another by a bifurcation. That is, the periodic dynamics of $g$ are \textbf{uncollapsible}.
    \item Finally, there exists some $N>0$ and some shift-invariant $\Sigma\subseteq\{0,...,N\}^\mathbb{N}$ s.t. the dynamics of every map in the isotopy class of $h:S\to S$ can be factored to those of the one-sided shift $\sigma:\Sigma\to\Sigma$.
\end{itemize}
\end{Theorem}

According to Theorem \ref{stability} the dynamical complexity of a pseudo-Anosov map in some isotopy class serves as a "lower bound" for the dynamical complexity of any $h$ isotopic to it - in other words, the dynamics of pseudo-Anosov maps are \textbf{dynamically minimal}. We are now ready to state the Thurston-Nielsen Classification theorem $[32]$,  the Betsvina-Handel algorithm $[35]$ - which, for convenience, we state as one result:
\begin{Theorem}
    \label{betshan} Let $f:\mathbb{D}\to\mathbb{D}$ denote some orientation-preserving disc homeomorphism and let $\gamma$ denote a braid corresponding to some finite collection of periodic orbits for $f$. Then, the braid type of $\gamma$ can be classified as follows:
    \begin{itemize}
        \item If there exists some $n>0$ s.t. $f^n$ is homeomorphic to the identity, we say $\gamma$ is a \textbf{finite--order braid}.
        \item If there exists some pseudo-Anosov map $g$ isotopic to $f$ in $\mathbb{D}\setminus\gamma$, we say $\gamma$ is a \textbf{pseudo--Anosov braid}. 
        \item We say $\gamma$ is a \textbf{reducible braid} if there exists a map $g$, isotopic to $f$ in $\mathbb{D}\setminus\gamma$ and a collection of arcs $S$ satisfying:
        \begin{enumerate}
            \item $g(S)=S$.
            \item Every component of $(\mathbb{D}\setminus\gamma)\setminus S$ is homeomorphic to a disc punctured in at least two points.
            \item In every component of $(\mathbb{D}\setminus\gamma)\setminus S$ $g$ is either pseudo--Anosov or finite--order.
        \end{enumerate}

    \end{itemize}
\end{Theorem}
In other words, Theorem \ref{betshan} implies two things:
\begin{itemize}
    \item Given any finite collection $\gamma_1,...,\gamma_2$ of periodic orbits for $f$, we can glue them to a braid $\gamma$ in some canonical way.
    \item The isotopy class of $f$ in $\mathbb{D}\setminus\gamma$ always includes some dynamically minimal map whose dynamics are precisely those forced by the braid $\gamma$. In particular, when $\gamma$ is either pseudo-Anosov or reducible, it forces complex dynamics which persist throughout the isotopy class of $f$ in $\mathbb{D}\setminus\gamma$.\\
\end{itemize}

Before concluding this section, we recall how the ideas mentioned above can be translated into the language of complex structures. To do so, let $D\subseteq\mathbb{C}$ be a domain, let $f:D\to\mathbb{C}$ be an orientation-preserving homeomorphism, and define $f_z=\frac{1}{2} (f_x-if_y)$, $f_{\overline{z}} =\frac{1}{2} (f_x+if_y)$ (in the distributional sense). We say $f$ is \textbf{$K$-quasiconformal} (or in short, $K-$Q.C.) for some $K>1$ if it satisfies the following:

\begin{itemize}
    \item $f$ is absolutely continuous on a.e. horizontal and vertical lines in $D$.
    \item $|f_{\overline{z}}|\leq k|f_{z}|$ a.e. in $D$, where $k=\frac{K-1}{K+1}$.
\end{itemize}

When $f$ is a diffeomorphism, it is easy to see that this condition is equivalent to $\frac{|f_z|+|f_{\overline{z}}|}{|f_z|-|f_{\overline{z}}|} \leq K$, see $[30]$ for precise details. Recall we have the following result, often called the Measurable Riemann Mapping theorem:
\begin{Theorem}
    \label{beltrami} Given any $\mu\in L^\infty(\mathbb{C})$ s.t. $||\mu||_\infty<1$, there exists a unique Q.C. homeomorphism $f:\overline{\mathbb{C}}\to\overline{\mathbb{C}}$ satisfying:
    \begin{itemize}
        \item $\frac{f_{\overline{z}}}{f_z}=\mu$ a.e. in $\overline{\mathbb{C}}$ - i.e., there exists a solution to the Beltrami Equation defined by $\mu$.
        \item $f$ can be chosen to fix three prescribed points in $\overline{\mathbb{C}}$: $f(0)=0$, $f(1)=1$, $f(\infty)=\infty$. This choice of $f$ is unique.
    \end{itemize}
    In particular, if $\frac{f_{\overline{z}}}{f_z}=\frac{g_{\overline{z}}}{g_z}$ a.e. for some other homeomorphism $g:\overline{\mathbb{C}}\to\overline{\mathbb{C}}$, then there exist Möbius transformations $M_1$ and $M_2$ s.t. $f=M_1\circ g\circ M_2$.
\end{Theorem}

This leads us to ask the following: since braids can force complex dynamics to appear, what is the connection between that dynamical complexity and the dilatation of $f:\mathbb{D}\setminus\gamma\to\mathbb{D}\setminus\gamma$? As proven in $[29]$, we have the following answer to both these questions:

\begin{Theorem}
    \label{bers} We can further choose the dynamicaly minimal map $g$ given by Theorem \ref{betshan} s.t. its dilatation is minimal in its isotopy class in $D$. In detail, setting $\mu=\frac{g_{\overline{z}}}{g_z}$, we can choose $g$ to satisfy the following:
    \begin{itemize}
        \item When the braid is finite order, we can choose $\mu$ to vanish a.e. in $D$ - i.e., $g$ can be chosen to be conformal.
        \item When the braid is pseudo-Anosov, $||\mu||_\infty\in(0,1)$ - in particular, $g$ cannot be chosen to be conformal in $D$.
        \item When the braid is reducible, the dilatation in every component of $D\setminus\cup_{i=1}^nC_i$ is also minimal.
        \item The quantity $||\mu||_\infty$ is minimal w.r.t. dilatations in the isotopy class of $f$.
    \end{itemize}
\end{Theorem}

\section{Braiding routes to chaos}
\label{rtc}
In this section, we present our results. We begin by rigorously translating the notions of period-doubling routes to chaos using the language of braids, after which we derive a formal invariant, $V$, of these routes to chaos. Following that, in Subsection \ref{qci} we transform $V$ into the \textbf{Conformal Index} (see Definition \ref{chaoticdil}). Briefly speaking, it describes how much the geometric model for the "endpoint" of the cascades deviates from being a conformal map - and then we apply it to study the bifurcations leading to the creation of horseshoe maps (see Theorem \ref{ch} and Theorem \ref{PA}) and complex behavior (see Corollary \ref{henoncor} and Corollary \ref{shilPA}).\\

Further, in Subsection \ref{numerin} we show how the invariant $V$ can be transformed into two numerical invariants, the \textbf{Index-Invariant} and the \textbf{Trace-Invariant}, using the language of representation theory (Proposition \ref{num1}, Proposition \ref{num2} and Theorem \ref{numinv}). As some of the properties of these invariants depend on many unknowns which are beyond the scope of this paper, we end Subsection \ref{numerin} with several conjectures how our ideas can be extended. Finally, we conclude this section with Subsection \ref{numerin2} where we discuss how the algebraic invariants presented in Subsection \ref{numerin} can possibly be numerically approximated.\\

We are now ready to begin. We say a closed topological disc $ABCD\subseteq \mathbb{R}^2$ is a \textbf{topological rectangle} if it has four prescribed sides on its boundary: $AB,BD,CD$, $AC$. Following the notion of a Smale horseshoe map (see $[36]$ and Subsection \ref{topod}), we now define the notion of a deformed horseshoe:

\begin{Definition}
\label{tophors}    Let $H:ABCD\to\mathbb{R}^2$ be an orientation-preserving homeomorphism s.t. the invariant set of $H$ inside $ABCD$ is non-empty. We say $H$ is a \textbf{deformed horseshoe on $k$ symbols}, $k\geq2$, if the following is satisfied (see the illustration in Figure \ref{stab2}):

    \begin{itemize}
        \item  $H$ extends to a homeomorphism of $\mathbb{R}^2$.
        \item $H$ can be isotoped to a pseudo-Anosov map $G:\mathbb{R}^2\to \mathbb{R}^2$ s.t. the dynamics of $H$ on its invariant set in $ABCD$ is deformed isotopically as follows:
    \begin{enumerate}
        \item The map $G:ABCD\to\mathbb{R}^2$ is a Smale horseshoe map on $k$ symbols.
        \item As we isotope $G:ABCD\to\mathbb{R}^2$ back to $H:ABCD\to\mathbb{R}^2$, every component of the invariant set of $G$ in $ABCD$ is deformed into a unique component in the invariant set of $H$ in $ABCD$.
        \item All the periodic orbits of $G$ in $ABCD$ persist as we isotope $G:ABCD\to\mathbb{R}^2$ back to $H:ABCD\to\mathbb{R}^2$ - without changing their minimal periods (by the requirement above, they cannot collapse into one another).
    \end{enumerate}
\end{itemize}
\end{Definition}
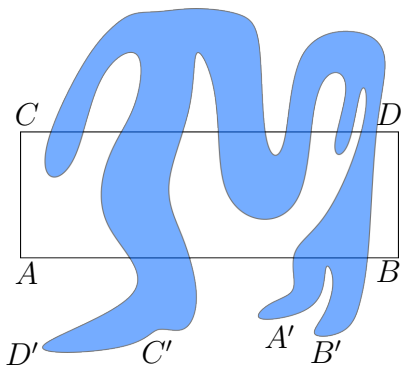
\begin{figure}[h]
\centering
\begin{overpic}[width=0.3\textwidth]{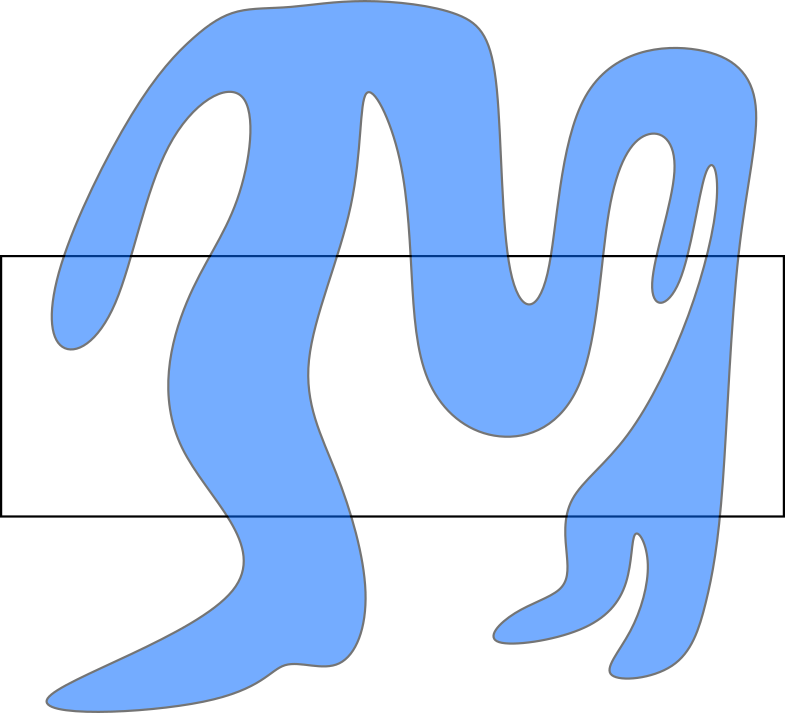}
\put(940,185){$B$}
\put(645,10){$A'$}
\put(765,-30){$B'$}
\put(940,600){$D$}

\put(-15,180){$A$}
\put(320,-30){$C'$}
\put(-40,-35){$D'$}
\put(-15,600){$C$}
\end{overpic}
\caption{\textit{A deformed $\cap$--horseshoe (in particular, this is a deformed horseshoe on $2$ symbols).}}\label{stab2}

\end{figure}

In this paper, we will be mostly interested in how the complex dynamics of deformed horseshoes form. Inspired by the ideas originally introduced in $[27,33]$ we study this problem by reducing it to a braid-theoretic question. In particular, recalling the connection between periodic orbits and braids (see Subsection \ref{topod}), from now on in this paper we always treat periodic orbits and braids as one and the same.\\

To begin, we first recall the notion of a period-doubling bifurcation and the fact that any periodic orbit can be identified with a braid. Following the definitions of $[27]$, let us consider an isotopy $f_t:ABCD\to\mathbb{R}^2$, $t\in[0,1]$ s.t. $\{x_1,...,x_k\}$ is a periodic orbit for $f_{t_0}$ with a braid type $\gamma$ (for some $t_0\in[0,1)$). We say $\{x_1,...,x_k\}$ undergoes a \textbf{period doubling bifurcation at $t_0$} provided the following holds (see the illustration in Figure \ref{cable}):

\begin{itemize}
    \item There exists some $\epsilon>0$ s.t. $\{x_1,...,x_k\}$ persists with the same minimal period for all $f_t$, $t\in(t_0-\epsilon,t_0]$.
    \item As we cross from $t<t_0$ to $t>t_0$, $\{x_1,...,x_k\}$ splits into at least two different periodic orbits, $x'$ and $x''$, defined as follows:
    
\begin{enumerate}

    \item The minimal period of $\{x'_1,...,x'_k\}$ is $k$ and its braid type is $\gamma$.
    
    \item The minimal period of $\{x''_1,...,x''_{2k}\}$ is $2k$ with  the braid type equal to $\gamma'$. By Proposition 4.1.2 in $[75]$ we know the braid type of $\{x''_1,...,x_{2k}\}$ is a "cable" of $\gamma$, see the illustration in Figure \ref{cable}.
\end{enumerate}
\end{itemize}

\begin{figure}[h]
\centering
\begin{overpic}[width=0.42\textwidth]{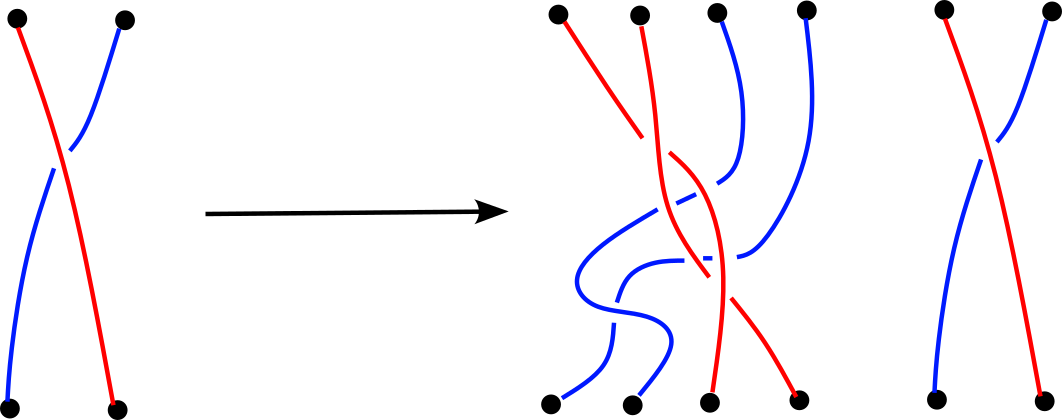}

\end{overpic}
\caption{\textit{A periodic orbit which is doubled - this causes its braid to "cable" on itself on the left, while keeping a copy of the initial braid on the right. }}\label{cable}

\end{figure}

We hightlight here that the braid "cabling" given by period-doubling bifurcations is probably not the same as the one described at $[14]$ and illustrated in Figure \ref{cab}.\\

In this paper, we will treat period-doubling bifurcations as the building blocks of complex dynamics. As stated in the introduction, this idea is motivated by how complex dynamics often appear after period-doubling cascades (see $[27,69,28]$ among others). As we plan to study this notion of "route to chaos" as a topological object, we now define it:
\begin{Definition}
\label{route}   Let $f_t:ABCD\to\mathbb{R}^2$ denote an isotopy, $t\in[0,1]$ satisfying the following:

\begin{itemize}
    \item For all $t$, $f_t$ is orientation-preserving.
    \item $f_1:ABCD\to\mathbb{R}^2$ has some invariant set $I$, s.t. $I$ includes infinitely many periodic orbits. 
    \item Let $P$ denote the collection of periodic orbits in $I$, and consider the set $Per\subseteq \mathbb{R}^2\times[0,1]$ s.t. $(x,t)\in Per$ precisely when $x$ is periodic for $f_t(x)$. Then, there exists a countable collection $\{C_i\}_{i \in \mathbb{N}}$ of period doubling cascades in $Per$ satisfying the following:
\begin{enumerate}
    \item Every cascade $C_i$ is defined by an initial periodic orbit $\gamma_i$, and an increasing sequence $\{t^i_n\}_{n\geq0}\subseteq[0,1]$, where $\gamma_i$ undergoes successive period-doubling bifurcations.
    \item We have $t^i_n\to t^i_\infty\leq1$. 
    \item Every $C_i$ defines a sequence of periodic orbits $\{\gamma^j_i\}_{j\geq0}\subseteq P$, s.t. each $\gamma^{j+1}_{i}$ is obtained from $\gamma^j_i$ by a period-doubling bifurcation at $t^i_{j+1},j\geq0$. In particular, $\gamma^j_i$ persists in $[t^i_j,1]$ without changing its braid type, and its minimal period is $2^jk$ (where $k\geq1$ is the minimal period of $\gamma_i=\gamma^0_i$).
    \item The sequences $\{\gamma^j_i\}_{j\in\mathbb{N}}$ are pairwise disjoint and maximal - i.e., $C_i$ is not a subcascade, and different cascades define different components of $Per$. Moreover, the braid types in $P$ are precisely $\cup_i\{\gamma^j_i\}_{j\in\mathbb{N}}$.

\end{enumerate}
\end{itemize}

The collection $C=\{C_i\}_{i \in \mathbb{N}}$ will be referred to as the \textbf{route to chaos leading to $I$}, or just the \textbf{route to chaos} when the set $I$ is clear from context.
\end{Definition}

The ordering of the sequence of cascades $C=\{C_i\}_{i\in\mathbb{N}}$ is just for the sake of formalism. As will be clear from our arguments, our results are independent of how we order the cascades inside $C$.\\

As shown in $[27]$, such routes to chaos are realized by models like the Henon map, among others (although, as in $[27]$ the said isotopies are required to be $C^1$, our scenario is more general). In this paper, we will be mostly interested in the topological properties of different routes to chaos. In order to study such properties, we now define the notion of "sameness" for different routes to chaos:

\begin{Definition}
    \label{type} Let $f_t:ABCD\to\mathbb{R}^2$ and $g_t:ABCD\to\mathbb{R}^2$, $t\in[0,1]$ be two isotopies defining respective routes to chaos, $C=\{C_i\}_{i \in \mathbb{N}}$, $C'=\{C'_j\}_{j \in \mathbb{N}}$. We say that the routes to Chaos $C$ and $C'$ are \textbf{topologically equivalent} or \textbf{have the same topological type} if there exists a continuous function $F:ABCD\times[0,1]\times[0,1]\to S$ satisfying the following:
    \begin{itemize}
        \item For all fixed $(t,s)\in[0,1]^2$, $F_{t,s}:ABCD\to\mathbb{R}^2$, $F_{t,s}(x)=F(x,t,s)$ is a homeomorphism. Moreover, for all $x\in ABCD$, $t\in[0,1]$, we have $F(x,t,0)=f_t(x)$ and $F(x,t,1)=g_t(x)$.
        \item There exists an increasing homeomorphism $r:[0,1]\to[0,1]$ s.t. for all fixed $t\in[0,1]$ the map $G_{r(t),s}:ABCD\to\mathbb{R}^2$, $G_{r(t),s}(x)=F(x,sr(t)+(1-s)t,s)$ is an isotopy, where $s\in[0,1]$ and $r(1)=1$, $r(0)=0$.
        \item For all $t\in[0,1]$, $g_{r(t)}$ and $f_t$ are isotopic away from $C\cap \mathbb{R}^2\times{t}$ and $C'\cap \mathbb{R}^2\times\{r(t)\}$. In other words, the maps are isotopic on the corresponding slices, away from the points on the cascades.
    \end{itemize}

    We will often call $F$ \textbf{an isotopy relative to $C$}. Similarly, we also say\textbf{ the isotopy $f_t:ABCD\mathbb\to\mathbb{R}^2$ can be deformed to $g_t:ABCD\to\mathbb{R}^2$ relatively to $C$}.
\end{Definition}

\begin{figure}[h]
\centering
\begin{overpic}[width=0.17\textwidth]{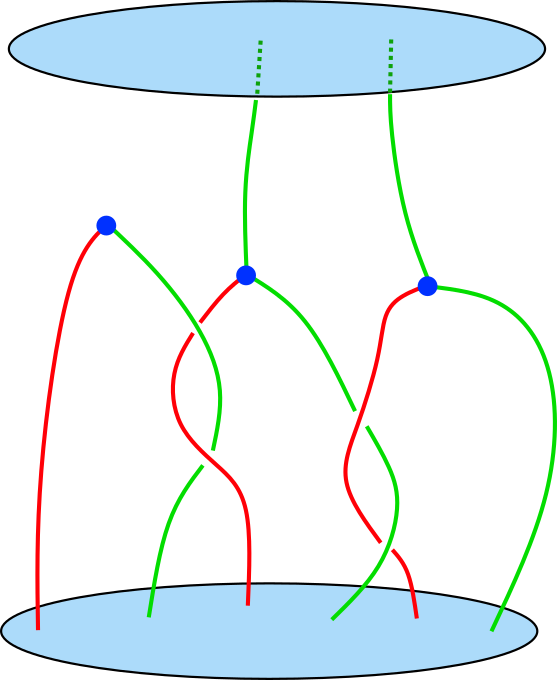}

\end{overpic}
\caption{\textit{An isotopy deforming a map on the top to a map on the bottom. Two periodic orbits for the upper map undergo period-doubling bifurcations, while two more orbits are added via a saddle node bifurcation. The curves of periodic points generate a "branched" braid connecting the lower and upper discs.}}\label{branched}

\end{figure}

From these definitions, inspired by the Mapping Class Group of a punctured disc, one easily concludes the following:

\begin{Corollary}
\label{bifurcationclass} Consider the collection of isotopies $f_t:ABCD\to\mathbb{R}^2$ which define the same route to chaos, $C$. Then, the deformations relative to $C$ define an equivalence relation on this collection. We will denote by ${BC}$, the \textbf{bifurcation class}, as the collection of the equivalence classes.
\end{Corollary}

At this point, we note that the set $BC$ is analogous to the Mapping Class Group of a disc in the sense that it can also be described using a topological construct similar to a braid. To illustrate, recall the Mapping Class Group of an $n$--punctured disc and the braid group $B_n$ are one and the same. Now, given an isotopy $f_t:ABCD\to\mathbb{R}^2$, $t\in[0,1]$ which defines a route to chaos $[C]$, define $P(C)=\{(x,t)|\; \exists n>0, f^n_t(x)=x\text{, and $x$ lies on some cascade in $C$}\}$. By definition, $Per$ is a collection of curves in $ABCD\times[0,1]$ that are branched precisely at the period-doubling bifurcation orbits (see the illustration in Figure \ref{branched}). Intuitively,  the equivalence class $[C]$ in $BC$ can be envisioned as any isotopy which can be deformed to $f_t:ABCD\to\mathbb{R}^2$ in $ABCD\times[0,1]\setminus P(C)$.\\

From now on, throughout this paper, we will be interested in the equivalence class $[C]$. Our approach will always be based on taking two-dimensional "slices" of the route to chaos, and treating it as a collection of braids which lie on the cascades composing $C$. This will allow us to define several topological invariants of $[C]$, which give a concrete description of its induced topological dynamics. We begin with the following formal result:
\begin{Theorem}
 \label{assol}   Given any period-doubling route to chaos $C=\{C_i\}_{i\in\mathbb{N}}$, we canonically associate a sequence of braid subgroups $V_i=\{\beta^n_i\}_{n\in\mathbb{N}}$ with it. Moreover, the collection $V= \{V_i\}_{i\in\mathbb{N}}$ is a topological invariant of $[C]$.
\end{Theorem}
\begin{proof}
Recall that each period-doubling cascade $C_i$ is defined by $\{\gamma_i,\{t^i_j\}_{j\in\mathbb{N}}\}$,  where $\gamma_i$ is the periodic orbit from which the cascade begins, and $\{t^i_j\}_{j\in\mathbb{N}}$ are the period-doubling times. From now on, we will always treat the orbit $\gamma_i$ (and any other periodic orbit) as a braid  on $n\geq1$ strands. Similarly, we will also denote by $\gamma^n_i$ the braid generated from $\gamma_i$ after $n$ period-doubling bifurcations. As stated at the beginning of the section, $\gamma^n_i$ is obtained from $\gamma_i$ by cabling it on itself $n$-times. To avoid confusion, we identify $\gamma^0_i$ with $\gamma_i$.

Now, further recall that once $\gamma^n_i$, $n\geq1$ is created as a braid type in $t^i_n\in[0,1)$, it persists as we vary $f_{t^i_n}$ to $f_1$ along the isotopy. It is easy to see that at the very least for $n\geq2$ the surface $S_n=\mathbb{R}^2\setminus\cup_{j=1}^n\gamma^j_i$ is homeomorphic to a disc punctured at $k$ points, $k\geq3$. Therefore, for all sufficiently large $n$, $S_n$ has a negative Euler characteristic, and we can apply the Betsvina-Handel algorithm to derive a canonical braid $\Gamma^i_n$ encoding the way $\gamma_i,...,\gamma^n_i$ permute the boundary components of $S_n$. In other words, we glue $\gamma_i,...,\gamma^n_i$ canonically into a braid that encodes the dynamics of the cascade $C_i$ at the $n$--th stage. This captures both the persistence of all the braid types that were already created (that is, $\gamma_i,...,\gamma^{n-1}_i$), and the braid type that was added at the latest bifurcation,  $\gamma^n_i$.

At this point we remark the only case in which $S_n$ does not have a negative Euler characteristic is when $n=1$ and $\gamma_i$ is a fixed point. In other words, the only case where $S_1$ does not have a negative Euler characteristic is when the original braid, $\gamma_i$, is a fixed point - and this changes immediately after $\gamma_i$ undergoes its first period-doubling bifurcation, as $\mathbb{R}^2\setminus(\gamma_i\cup\gamma^2_i)$ is homeomorphic to a disc with at least two punctures. Therefore, in order to avoid any degenerate cases, whenever $S_1$ does not have a negative Euler characteristic we always associate with it the trivial braid $B_1$ - i.e., the braid group on one strand.

To summarize, for all $n\geq1$ we can associate with $S_n$ a braid $\Gamma^i_n$ canonically. We now associate a braid subgroup with $\Gamma^i_n$, $\beta^i_n$. There are two cases we should consider - either $S_n$ has a negative Euler characteristic, or it does not. In the second case, we choose $\beta^i_1$ to be simply $B_1$, the trivial braid group on one strand. When $S_n$ has a negative Euler characteristic, by the Betsvina-Handel theorem (see Theorem \ref{betshan}) we know that $\Gamma^i_n$ is a braid on $r$ strands - where $r\leq (2^{n+1}-1)k$ and $k$ is the minimal period of $\gamma_i$. In this case, we choose $\beta^n_i=<\Gamma^i_n>$, where $<\Gamma^i_n>$ is the minimal braid subgroup of $B_r$, which includes $\Gamma^i_n$.

Since the braid given by the Betsvina-Handel algorithm is canonical, the same is true for the braid subgroup $\beta^n_i$. This, in turn, implies the sequence $V_i=\{\beta^n_i\}_{n\in\mathbb{N}}$ is also canonical - and hence so is the collection $V=\cup_i\{V_i\}$. As the braid types of periodic orbits on period doubling cascades in $C$ do not change under relative isotopies w.r.t. $C$, the sequence $V_i$ is a topological invariants of the cascade $C_i$, that remains invariant under isotopies relative to $C$. This proves that $V=\cup_i\{V_i\}$ also remains unchanged under relative isotopies of $C$.
\end{proof}

It is easy to visualize $V$ as the "skeleton" of the period-doubling route to chaos, which deforms $f_0$ to $f_1$ along the isotopy $f_t:ABCD\to\mathbb{R}^2, t\in[0,1]$. Now, consider any two isotopies $f_t:ABCD\to\mathbb{R}^2$ and $g_t:ABCD\to\mathbb{R}^2$, $t\in[0,1]$ which define the respective routes to chaos $C$ and $C'$. As an immediate consequence of Theorem \ref{assol}, we conclude: 

\begin{Corollary}
    Assume $V$ and $V'$ are the collections of braid groups corresponding to the routes to chaos $C$ and $C'$. Then, we have the following:

    \begin{itemize}
        \item If $V\ne V'$, the corresponding routes to chaos must include different braids. 
        \item If $V=V'$, then the cascades of $C$ differ from those of $C'$ at most by the order of appearance of periodic orbits along the isotopy. 
    \end{itemize}
\end{Corollary}

With previous notations, the equality $V'=V$ does not imply $[C']=[C]$. To see why, note the sets $ABCD\times[0,1]\setminus P(C')$ and $ABCD\times[0,1]\setminus P(C)$ do not even have to be homeomorphic.\\

The invariant $V$, despite its power, has one major drawback. Namely, given an isotopy $f_t:ABCD\to\mathbb{R}^2$ as above, it can be extremely hard to actually study $V$. To overcome this difficulty, in the next subsections we reduce $V$ to other invariants, the \textbf{Conformal Index} (see Definition \ref{chaoticdil}), the \textbf{Index-Invariant} and the \textbf{Trace-Invariant} (see Corollary \ref{sequence}). As will be made clear later on, these invariants can be analyzed either analytically or numerically, and all are invariants of the bifurcation class $[C]$. 

\subsection{The Conformal Index}
\label{qci}

In this subsection, we prove how the set $V$ defined above can be reduced to a collection of planar homeomorphisms (or alternatively, to some subset of $L^\infty(\mathbb{D})$, where $\mathbb{D}$ is the unit disc), a collection to which we will refer as the \textbf{Conformal Index} (see Definition \ref{chaoticdil}). We will do so using the theory of quasiconformal mappings introduced at the end of Subsection \ref{topod}. We will first define and analyze the Conformal Index, after which we will use it to study routes to chaos (see Theorem \ref{PA}). Following that, we will analyze two concrete examples of routes to chaos and show how the Conformal Index can be used to study them (see Corollary \ref{henoncor} and Corollary \ref{shilPA}).\\

We first recall our setting. As stated in the previous section, we always consider an isotopy $f_t:ABCD\to\mathbb{R}^2$, $t\in[0,1]$, which defines a route to chaos, i.e., a collection of period-doubling cascades $C=\{C_i\}_{i\in\mathbb{N}}$ (see Definition \ref{route}). Recall that we always identify the periodic orbits of $f_t:ABCD\to\mathbb{R}^2$, $t\in[0,1]$, with braids, and that each cascade $C_i=\{\gamma_i,\{t^i_n\}_{n\in\mathbb{N}}\}$ is defined by an initial braid $\gamma_i$ and some increasing sequence $\{t^i_n\}_{n\in\mathbb{N}}\subseteq[0,1]$ corresponding to the period-doubling bifurcations of $\gamma_i$ - that is, at $t^i_1$ the braid $\gamma_i=\gamma^0_i$ is bifurcated into $\gamma^1_i$, at $t^i_2$ the braid $\gamma^1_i$ is bifurcated into $\gamma^2_i$ and so on (see the illustration in Figure \ref{cable}).\\

We now begin. From now on, $\mathbb{D}$ would always denote the open unit disc. We first note that given any isotopy $f_t:ABCD\to\mathbb{R}^2$, $t\in[0,1]$ we can always extend it to an isotopy $F_t:\mathbb{D}\to \mathbb{D}$, $t\in[0,1]$ satisfying the following:

\begin{itemize}
    \item For all $t\in[0,1]$, $F_t$ and $f_t$ coincide on $ABCD$.
    \item If $f_t, t\in[0,1]$, generates a route to chaos, so does $F_t, t\in[0,1]$.
\end{itemize}

From now on, we denote by $M_\infty=\{\mu\in L^\infty(S^2): ||\mu||_\infty<1\}$ - we will think of $M_\infty$ as the collection of all complex dilatations. We first prove the following:
\begin{Proposition}
\label{coninf} For a route to chaos $C=\{C_i\}_{i\in\mathbb{N}}$ defined by $f_t:ABCD\to\mathbb{R}^2, \; t\in[0,1]$ we have:

\begin{itemize}
    \item Each cascade $C_i$ defines a non-empty collection of quasiconformal maps of $S^2$, denoted by $E(C_i)$, and their corresponding complex dilatations $Dil(C_i)\subseteq M_\infty$.
    \item The sets $Dil(C_i)$ and $E(C_i)$ are topological invariants of the cascade $C_i$ under relative deformations of the isotopy $f_t:ABCD\to\mathbb{R}^2$ w.r.t. $C$ (see Definition \ref{type}).
\end{itemize}
\end{Proposition}
\begin{proof}
Choose some $i\in\mathbb{N}$ and recall that we denote $C_i=\{\gamma_i,\{t^i_n\}_{n\in\mathbb{N}}\}$, where $\gamma_i$ is the initial braid that undergoes period-doubling bifurcation at the points $\{t^i_n\}_{n\in\mathbb{N}}\subseteq[0,1]$. Now, consider some extension of $f_t:ABCD\to\mathbb{R}^2$ to an isotopy $F_t:\mathbb{D}\to \mathbb{D}$ as described above, and consider the first $n$ orbits along the cascade $C_i$, $\gamma_i,\gamma^1_i,...,\gamma^n_i$. We now puncture the sphere $\mathbb{D}$ in these periodic orbits to derive a surface $S^i_n=\mathbb{D}\setminus(\gamma^n_i\cup...\cup\gamma_i)$.

 Let us consider the isotopy class of $F_t:S^i_n\to S^i_n$, $t\geq t^i_n$. Since $\gamma_i$ is a periodic orbit that undergoes  $n$ successive period-doubling bifurcations at $t^i_j$, $j=0,...,n$, we conclude that for any $n\geq1$ the set $S^i_n$ is homeomorphic to an open disc punctured at the very least at two interior points. Therefore, for all $n>0$ we can apply the Betsvina-Handel algorithm to $F_t:S^i_n\to S^i_n$, $t\in(t^i_n,t^i_{n+1}]$, and describe its isotopy class in $S^i_n$ with some braid $\Gamma^i_n$ (see Theorem \ref{betshan}). Due to the Betsvina-Handel algorithm we know that $\Gamma^i_n$ is independent of the extension $F_t$ chosen, where $t\in(t^i_n,t^i_{n+1}]$.

By Theorem \ref{bers} it follows that there exists a non-empty collection $E(\Gamma^i_n)\subseteq L^\infty({\mathbb{D}})$ of extremal quasiconformal homeomorphisms $g:S^i_n\to S^i_n$ in the isotopy class of $\Gamma^i_n$, s.t. each $g$ minimizes the $||.||_\infty$ norm of its complex dilatation (w.r.t. the isotopy class of $g$). This motivates us to define $Dil(\Gamma^i_n)=\{\mu\in L^\infty(\mathbb{D})|\; \exists g\in E(\Gamma^i_n),\mu=\frac{g_{\overline{z}}}{g_z}\text{ a.e}\}$ - it is easy to see $Dil(\Gamma^i_n)\subseteq M_\infty$.

To conclude the proof, define $E(C_i)=\cup_{n\in\mathbb{N}} E(\Gamma^i_n))$ and $Dil(C_i)=\cup_{n\in\mathbb{N}}Dil(\Gamma^i_n)$. By construction, both $E(C_i)$ and $Dil(C_i)$ are non-empty sets. Finally, note that both $E(\Gamma^i_n)$ and $Dil(\Gamma^i_n)$ are determined precisely by the braid $\Gamma^i_n$, and that deformations of the isotopy $f_t:ABCD\to\mathbb{R}^2$ relative to the route to chaos $C$ do not change the braid types on $C_i$. Thus, both $E(C_i)$ and $Dil(C_i)$ persist, unchanged, under deformations relative to $C$.
\end{proof}

Recall that given a route to chaos defined by period-doubling cascades $C=\{C_i\}_{i\in\mathbb{N}}$ we defined the invariant $V=\cup\{V_i\}_{i\in\mathbb{N}}$, where $V_i$ are the sequences of braid groups derived from $C_i$ (see Theorem \ref{assol}). It is easy to see that both $E(C_i)$ and $Dil(C_i)$ are analogous to the sequence $V_i$ (and can easily be derived directly from it). This leads us to ask the following: can we use complex-analytic tools to define an analogue for $V$? It turns out that the answer is positive and we give it in the following corollary of Proposition \ref{coninf}:
\begin{Corollary}
\label{chaoticdil1} Consider an isotopy $f_t:ABCD\to\mathbb{R}^2$, $t\in[0,1]$ which defines a route to chaos $C=\{C_i\}_{i\in\mathbb{N}}$. Then, there exists a collection of quasiconformal maps $E(C)=\cup_i E(C_i)$, the \textbf{Extremal Invariant} and a set $Dil(C)={\cup_i Dil(C_i)}\subseteq M_\infty$, the \textbf{Dilatation Invariant}, which persist, unchanged, under deformations of the isotopy relative to $C$. Consequently, both are invariants of the bifurcation class (see Corollary \ref{bifurcationclass}).
\end{Corollary}

Corollary \ref{chaoticdil1} motivates us to introduce the following definition:

\begin{Definition}
    \label{chaoticdil} Consider an isotopy $f_t:ABCD\to\mathbb{R}^2$, $t\in[0,1]$ which defines a route to chaos $C=\{C_i\}_{i\in\mathbb{N}}$. Then, the pair $(E(C),Dil(C))$ will be referred to as the \textbf{Conformal Index}. By Corollary \ref{chaoticdil1} this index is invariant under deformations of the isotopy relative to $C$ - consequently it is an invariant of the bifurcation class. We say \textbf{the }\textbf{Conformal Index vanishes} precisely when $Dil(C)=\{0\}$ (i.e., when $E(C)$ includes only conformal maps).
\end{Definition}

Having proven the Conformal Index is an invariant of the route to chaos $C$, we now address the following questions: 

\begin{itemize}
    \item What are the dynamical implications of $Dil(C)\ne\{0\}$? Alternatively, when should we expect $E(C)$ to include only conformal maps?
    \item How sensitive is the Conformal Invariant? Or, in other words, when should we expect the Conformal Indices of two routes to chaos to differ?
\end{itemize}

As we will now see, the answer to the first question is surprisingly simple and has several interesting dynamical consequences. Using the answer to this question, we then study the uniqueness and sensitivity properties of the Conformal Index, thus giving partial answer to the second question.\\ 

We remark that intuitively, one can think of $Dil(C)$ as an object similar to entropy - that is a measure of the inherent complexity of the system. The the main difference between the two is that if entropy measures the complexity of a given map $f:\mathbb{D}\to\mathbb{D}$, $Dil(C)$ measures the complexity of the bifurcation structure of an isotopy $f_t:ABCD\to\mathbb{R}^2$. For example, when $Dil(C)=\{0\}$, by Theorem \ref{bers} we would expect the route to chaos $C$ to consist only of finite-order braids, which, by Theorem \ref{bers}, correspond to conformal maps of the disc on itself (that is, Möbius transformations). Conversely, when $Dil(C)\ne 0$ we would expect the exact opposite, namely, the existence of braids in $C$ that are pseudo-Anosov and force complex dynamics to appear (see Theorems \ref{stability}, \ref{betshan}). We make this intuition rigorous by proving:
\begin{Theorem}
    \label{ch} Let $C=\{C_i\}_{i\in\mathbb{N}}$, $C_i=\{\gamma_i,\{t^i\}_{n\in\mathbb{N}}\}$ be a route to chaos, realized by some isotopy $f_t:ABCD\to\mathbb{R}^2$, $t\in[0,1]$. Then, if $Dil(C)\ne0$, we have the following:
    
  \begin{itemize}
        \item There exists some $1>t_0\geq0$ s.t. for all $r\geq t_0$, the homeomorphism $f_{r}:ABCD\to \mathbb{R}^2$ has infinitely many periodic orbits.  
        \item  A necessary and sufficient condition for the Conformal Index to vanish is that for all $i,n>0$ the braid $\Gamma^i_n$ is a finite order braid. 
\end{itemize}
\end{Theorem}
\begin{proof}
Given a route to chaos $C$ whose Conformal Index is non-zero, it follows $Dil(C)\ne\{0\}$ - i.e., $Dil(C)$ includes some non-zero complex dilatations. This implies there exists some $i\in\mathbb{N}$ s.t. $Dil(C_i)\ne0$. Therefore, for the cascade $C_i$ the Extremal Invariant $E(C_i)$ includes a non-conformal map $f$ with a dilatation $\mu\in Dil(C_i)$ s.t. $1>||\mu||_\infty>0$.

In detail, using previous notations (see Proposition \ref{coninf}), the map $f:S^i_n\to S^i_n$ minimizes the dilatation in the isotopy class of $f_t:S^i_n\to S^i_n$ (where $t\in(t^i_n,t^i_{n+1})$ and $S^i_n=\mathbb{D}\setminus\{\gamma_i,...,\gamma^i_n\}$). Recalling Theorem \ref{bers} (and that quasiconformal maps are extendable over singletons) we conclude that $f:\mathbb{D}\to \mathbb{D}$ is $K$--Quasiconformal, for some $K>1$. This implies that the Betsvina-Handel algorithm glues $\gamma_i,\gamma^1_i,...,\gamma^n_i$ to a braid $\Gamma^i_n$ which is either a pseudo-Anosov braid or a reducible braid with a pseudo-Anosov component.

By Theorems \ref{stability}, \ref{betshan} it follows that any homeomorphism $G:\mathbb{D}\to \mathbb{D}$, which generates the braids $\gamma_i,\gamma^1_i,...,\gamma^n_i$ also has to generate infinitely many periodic orbits - moreover, these orbits are all isotopy stable under isotopies of $G$ in $S^i_n$. Since by Definition \ref{route} we know that for all $t>t^i_n$ the homeomorphism $f_t:ABCD\to\mathbb{R}^2$ generates the braids $\gamma_i,\gamma^1_i,...,\gamma^n_i$ as periodic orbits, we conclude that for all extensions of the isotopy $f_t:ABCD\to\mathbb{R}^2$ to an isotopy $F_t:\mathbb{D}\to \mathbb{D}$, $t\in(t^i_n,1]$, $F_t$ generates infinitely many periodic orbits.

We now prove that given $t>t^i_n$, the periodic orbits forced by $\Gamma^i_n$ all lie on $ABCD$. To do so, consider $D_n=\overline{\mathbb{D}}\setminus\{\gamma_i,\gamma^1_i,...,\gamma^n_i\}$ and a homeomorphism $G:D_n\to D_n$, $t\in[0,1]$ s.t. the following holds:

\begin{itemize}
    \item $D_n\setminus (ABCD\cup S^1)=\mathbb{A}$ is an open annulus, foliated by non-intersecting curves $\{L^y_x\}_{x\in\partial ABCD}$, where each $L^y_x$ connects a unique $y\in S^1$ to a unique $x$ in $ABCD$. 
    \item     On the (closed) rectangle $ABCD$ the map $G$ coincides with $f_t$. 
    \item We have $G_(L^y_x)=L^{G(y)}_{G(x)}$. Moreover, $y$ is periodic precisely when $x$ is, with the same minimal period.
    \item For every $z\in L^y_x$, $\lim_{k\to\infty}d(G^k(y),G^k(x))=0$, where $d$ is the Euclidean metric on $D$.
\end{itemize}

By the definition of $G$, the open annulus $\mathbb{A}$ does not include any periodic orbits. Therefore, since $G$ generates the braid $\Gamma^i_n$ and the periodic dynamics associated with it, all the dynamics forced by $\Gamma^i_n$ in $D_n$ lie in $ABCD$. However, since $G$ and $f_t$ coincide on the closed rectangle $ABCD$, we conclude that $f_t:ABCD\to\mathbb{R}^2$ also generates all infinitely many periodic orbits forced by the braid $\Gamma^i_n$. As $t>t^i_n$ was chosen arbitrarily, the first assertion of the theorem now follows.
 
We now prove the second assertion of the theorem - that is, we prove a necessary and sufficient condition for $Dil(C)=\{0\}$ is that for all $i$ and all $n$ the braid $\Gamma^i_n$ is finite order. By the above it is clear that a sufficient condition for all $\Gamma^i_n$ to be finite order is $Dil(C)=\{0\}$. The reason this is so is because otherwise, the same argument implies there exists some $i,n$ for which the braid $\Gamma^i_n$ is either pseudo-Anosov or reducible with a pseudo-Anosov component. Therefore, a sufficient condition for all braids $\Gamma^i_n$ to be finite order is that $Dil(C)=\{0\}$, i.e., that there are no proper quasiconformal maps in $E(C)$.

We now show that when all $\Gamma^i_n$ are finite order, the only possibility is $Dil(C)=\{0\}$. To this end, we recall that for any $i,n$, the set $Dil(\Gamma^i_n)$ denotes the complex dilatations of the extremal maps corresponding to $\Gamma^i_n$ given by Theorem \ref{bers}. By Theorem \ref{bers} it follows that if for all $i,n$ the braid $\Gamma^i_n$ is finite order, then $Dil(\Gamma^i_n)=\{0\}$. Therefore, by $Dil(C_i)=\cup_{i,n}Dil(\Gamma^i_n)$ and $Dil(C)=\cup_i Dil(C_i)$ we see that when the braids $\{\Gamma^i_n\}_{i,n\in\mathbb{N}}$ are all finite order, the only possibility is $Dil(C)=\{0\}$.
\end{proof}

In fact, we can even sharpen the result of Theorem \ref{ch} a little further. To do so, let $\Lambda$ be a disc punctured at $k$ points and let $f:\Lambda\to \Lambda$ be a homeomorphism defining either a pseudo-Anosov or a reducible braid. As proven in $[35]$, whenever this is the case, there exists some $\Sigma\subseteq\{1,...,k-1\}^\mathbb{N}$, invariant under the one-sided shift $\sigma:\{1,...,k-1\}^\mathbb{N}\to\{1,...,k-1\}^\mathbb{N}$, and some $f$--invariant $I\subset\Lambda$ and a continuous, surjective $\pi:I\to\Sigma$ s.t. the following holds:

\begin{itemize}
\item$\pi\circ f=\sigma\circ\pi$. Moreover, the set $I$ corresponds to the dynamics forced by the braid.
    \item There exists a dense collection of periodic orbits in $\Sigma$.
    \item If $s\in\Sigma$ is periodic of minimal period $k$ for $\sigma$, then $\pi^{-1}(s)$ includes a periodic point of minimal period $k$ for $f$.
    \item There exists a continuous interval map $g:[0,1]\to[0,1]$ with critical points $c_1,...,c_r$ and a continuous surjection $\nu:I\to I'$ s.t. $\nu\circ f=g\circ\nu$ - where $I'$ is the invariant set of $g$ in $[0,1]\setminus\{c_1,...,c_r\}$.
\end{itemize}

As shown during the proof of Theorem \ref{ch}, whenever the Conformal Index is non-zero, there exists some $t_0\in[0,1)$ s.t. for all $t>t_0$ the map $f_t$ satisfies two things:
\begin{itemize}
    \item $f_t$ has a collection of periodic orbits in $ABCD$ forming a braid that is \textbf{not} finite order.
    \item All the dynamics forced by the said braid are trapped inside the rectangle $ABCD$.
\end{itemize}

This allows us to sharpen the conclusion of Theorem \ref{ch} as follows:

\begin{Corollary}
    \label{subshift} Let $f_t:ABCD\to\mathbb{R}^2$, $t\in[0,1]$ be an isotopy defining a route to chaos $C$ s.t. $Dil(C)\ne0$. Then, there exists some $t_0\in[0,1)$ s.t. for all $t>t_0$ the following holds:
\begin{itemize}
    \item There exists an $f_t$ invariant set $I\subseteq S$, some $n>0$ and a $\sigma$--invariant $\Sigma\subseteq\{1,...,n\}^\mathbb{N}$ along with a continuous, surjective $\pi:I\to\Sigma$ w.r.t. which $\pi\circ f_t=\sigma\circ\pi$.
    \item There exists a countable, dense collection of periodic orbits in $\Sigma$.
    \item If $s\in\Sigma$ is periodic, $\sigma^{-1}(s)$ includes at least one periodic point for $f_t$ - of the same minimal period.
    \item There exists a one-dimensional map $g:[0,1]\to[0,1]$ with critical points $c_1,...,c_r$ whose dynamics on its invariant set $[0,1]\setminus\{c_1,...,c_r\}$ are a factor map of $f_t$ on $I$.
\end{itemize}

\end{Corollary}
\begin{figure}[h]
\centering
\begin{overpic}[width=0.5\textwidth]{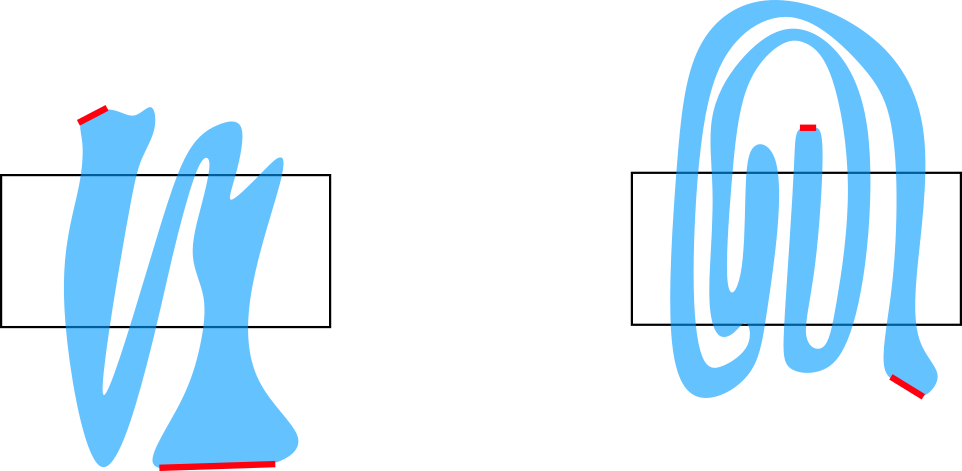}
\put(640,100){$A$}
\put(980,100){$B$}
\put(640,315){$C$}
\put(980,315){$D$}

\put(-15,100){$A$}
\put(320,100){$B$}
\put(320,315){$D$}
\put(-15,315){$C$}
\end{overpic}
\caption{\textit{Two deformed horseshoes which cannot be deformed to one another on their invariant set in $ABCD$, as the one on the left corresponds to the shift on three symbols, while the one on the right corresponds to the shift on five symbols. The red arcs are the images of the $AB$ and $CD$ sides.}}\label{diffe}

\end{figure}
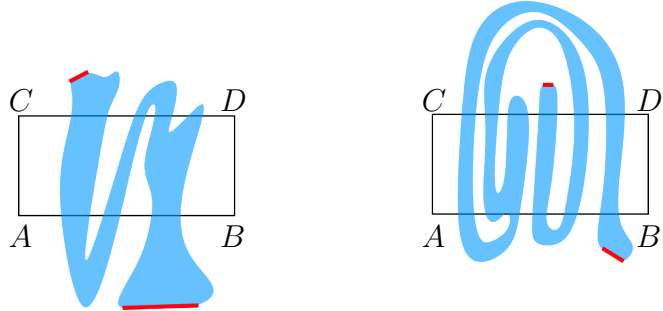

Informally, Corollary \ref{subshift} implies that whenever $Dil(C)\ne0$ there exists some "topological lower bound" for the dynamical complexity of $f_t$, $t>t_0$, given by the subshift of finite type $\sigma:\Sigma\to\Sigma$ described above (or alternatively, by $g:[0,1]\to[0,1]$). That being said, the subshift $\Sigma$ need not necessarily be unique, and the "topological lower bound" can probably be tightened as $t\to1$. Regardless, the main takeaway from this discussion is that when $Dil(C)$ is non-zero, the dynamical complexity of $f_t:ABCD\to\mathbb{R}^2$ only increases as $t\to 1$.\\

We now use our results to study the following question: let $ABCD\subseteq \mathbb{R}^2$ be some topological rectangle, and let $f_t:ABCD\to\mathbb{R}^2$, $t\in[0,1]$ be an isotopy defining a route to chaos $C$ s.t. $f_1$ is a deformed horseshoe (as given by Definition \ref{tophors}). Then, should we expect $Dil(C)\ne\{0\}$? Alternatively, one could also ask a different question: assume $f_t,f'_t:ABCD\to\mathbb{R}^2$, $t\in[0,1]$ are isotopies defining respective routes to chaos $C$ and $C'$ s.t. $f_1,f'_1:ABCD\to\mathbb{R}^2$ define deformed horseshoes which \textbf{cannot} be deformed to one another on their invariant sets in $ABCD$ (see, for example, Figure \ref{diffe}). Then, should we expect $Dil(C)\ne Dil(C')$? We now answer both these questions:

\begin{Theorem}
    \label{PA} Let $C$ be a period-doubling cascade defined by an isotopy $f_t:ABCD\to\mathbb{R}^2$, $t\in[0,1]$, terminating in some deformed horseshoe $f_1:ABCD\to\mathbb{R}^2$. Then, the Conformal Index of $C=\{C_i\}_{i\in\mathbb{N}}$ is non-zero. Moreover, if $f_t,f'_t:ABCD\to\mathbb{R}^2$, $t\in[0,1]$ are isotopies defining respective routes to chaos $C$ and $C'$ s.t.:
        
        \begin{enumerate}
            \item  $f_1$ and $f'_1$ are deformed horseshoes as in Definition \ref{tophors},
            \item $f_1$ and $f'_1$ cannot be deformed on their invariant sets to the same Smale horseshoe,
        \end{enumerate}
      then we have $E(C)\ne E(C')$ - i.e., the Conformal Index distinguishes between $C$ and $C'$.
    
\end{Theorem}
\begin{proof}
We first prove that given an isotopy $f_t:ABCD\to\mathbb{R}^2$, $t\in[0,1]$ which defines a route to chaos $C$ terminating in a deformed horseshoe, the Conformal Index of $C$ is non-zero. We do so by applying convergence results from the theory of quasiconformal mappings. As a consequence, we prove that if $f_t,f'_t:ABCD\to\mathbb{R}^2$ are isotopies,defining the routes to chaos $C$ and $C'$, then their conformal indices differ.

To prove the non-vanishing properties of the Conformal Index, recall that given a Smale horseshoe map $H:ABCD\to \mathbb{D}$, there exists a dense orbit for $H$ in its invariant set in $ABCD$. In particular, recall that this orbit can be approximated by a sequence of periodic orbits $\{\delta_j\}_{j\in\mathbb{N}}$ for $H$, of increasingly large minimal periods (see the construction of the Smale horseshoe map in $[36]$).

Since $f_1$ is a deformed horseshoe, we know all the braids corresponding to the said orbits eventually appear in $C$ on some cascade. Therefore, for each $j$, set $\mu_j\in Dil(C)$ as the dilatation of the map $f_j\in E(C)$ corresponding to some braid $\Gamma^i_{n(j)}$, s.t. $\delta_j\in\{\gamma_i,...,\gamma^{n(j)}_i\}$, where $n=n(j)$ is always chosen to be minimal. We now prove there exists some $K$ s.t. the dilatations of $f_j$ is uniformly bounded by $K$.\\

To see why, consider a smooth isotopy $G_t:\mathbb{D}\to \mathbb{D}$, $t\in[0,1]$ satisfying the following:
\begin{itemize}
    \item There exists a rectangle $ABCD\subseteq S^2$ s.t. $G_1:ABCD\to S^2$ can be deformed on its invariant set to $f_1:ABCD\to \mathbb{R}^2$.
    \item $G_1:ABCD\to S^2$ is a Smale horseshoe, hyperbolic on its invariant set.
    \item For $t\in[0,1]$ sufficiently smaller than $1$, $G_t(ABCD)\cap ABCD=\emptyset$.
    \item For all $t$, $G_t$ extends smoothly over $S^1$, and moreover, $G_t(S^1)=S^1$. This implies there is some $K>1$ s.t. the dilatation of $G_t$ is bounded by $K$, where $K$ is independent of $t$.
\end{itemize}

Now, let $\{t_j\}_{j\in\mathbb{N}}$ denote some sequence in $[0,1]$ s.t. for all $t>t_j$, $G_{t}:ABCD\to \mathbb{D}$ generates the braids $\Gamma^i_{n(j)}$ (since for all $j$ the braid $\Gamma^i_{n(j)}$ exists in $ABCD$ w.r.t. both $f_1$ and $G_1$, there exists such $t_j$). We now consider the isotopy class of $G_{t_j}$ in $\mathbb{D}\setminus\Gamma^i_{n(j)}$. By Theorem \ref{bers} the dilatation of $f_j$ minimizes the dilatation in its isotopy class - hence, since $G_{t_j}$ and $f_j$ are in the same isotopy class, it follows the dilatation of $f_j$ is also bounded by $K$. Moreover, since $j$ was chosen arbitrarily, we know the dilatation of all maps in $\{f_j\}_{j\in\mathbb{N}}$ is bounded by $K$. 

At this point we note that since quasiconformal maps all extend over singletons on the boundary, we may further assume $\{f_j\}_{j\in\mathbb{N}}$ is a sequence of $K$--Quasiconformal disc maps. We now recall the following result from $[31]$:

\begin{Theorem}
    \label{lehto} Let $\{f_n\}_{n\in\mathbb{N}}$ be a family of $K'$--Quasiconformal homeomorphisms of some bounded domain $D$, s.t. $\cup_j f_j(D)$ misses three values of $S^2$. Then, there exists some subsequence $\{f_{n_k}\}_{k\in\mathbb{N}}$ which converges locally-uniformly in $D$ to some $K'$--Quasiconformal map $f:D\to D$.
\end{Theorem}
As for each $j$ the set $f_j(\mathbb{D})=\mathbb{D}$ and since for all $j$ the mapping $f_j$ is at most $K$--Quasiconformal, by Theorem \ref{lehto} we conclude there exists a limit function $f:\mathbb{D}\to \mathbb{D}$, which is $K$--Quasiconformal. Without loss of generality (and possibly by moving to a subsequence) we assume $f$ is the only limit function for $\{f_j\}_{j\in\mathbb{N}}$ in $\mathbb{D}$.

\begin{figure}[h]
\centering
\begin{overpic}[width=0.65\textwidth]{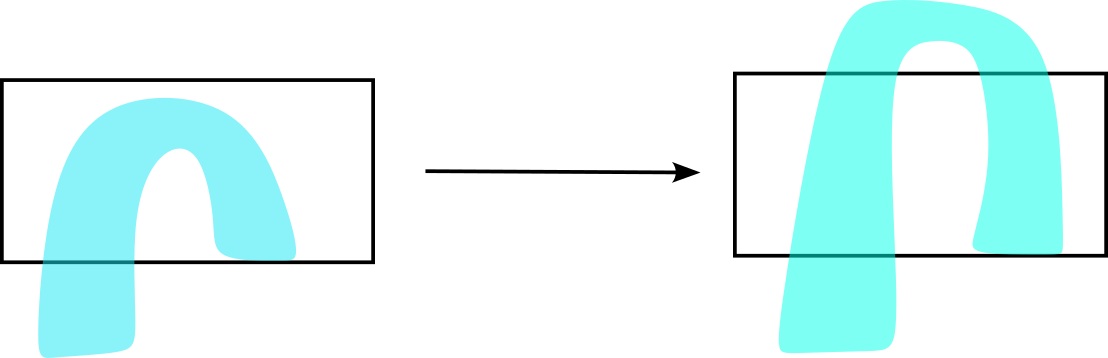}
\put(640,55){$A$}
\put(980,55){$B$}
\put(855,55){$A'$}
\put(915,55){$B'$}
\put(765,-30){$C'$}
\put(675,-30){$D'$}
\put(640,265){$C$}
\put(980,265){$D$}

\put(-15,50){$A$}
\put(320,50){$B$}
\put(180,50){$A'$}
\put(250,50){$B'$}
\put(90,-25){$C'$}
\put(10,-35){$D'$}
\put(320,265){$D$}
\put(-15,265){$C$}
\end{overpic}
\caption{\textit{As $f_j$ (on the left) tends to $f_\infty$ (on the right), the dense orbit determines the image of the sides $AB,AC,BD$ and $CD$ (where $A',B',C'$ and $D'$ denote the images of the respective vertices).}}\label{approxx}
\end{figure}
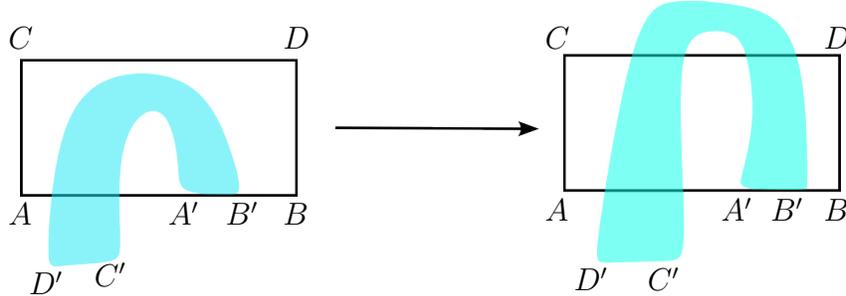

Assume by contradiction that $f$ is a $1$-Quasiconformal mapping of the disc $\mathbb{D}$. As $1$--Quasiconformal mappings are conformal, $f$ has to be a Möbius transformation which maps the disc to itself. Now, let us study the iterations of $f$ in $\mathbb{D}$. By the Wolf-Denjoy theorem $[82]$ we know that the dynamics of $f$ in $\mathbb{D}$ are either conjugate to a rotation of $\mathbb{D}$ or there exists some point $\zeta\in S^1$ s.t. for all $z\in\mathbb{D}$ we have $f^n(z)\to\zeta$. We will generate a contradiction by proving  both are impossible, i.e., we show that $f$ cannot satisfy the conclusions of the Wolf-Denjoy theorem.

For this we note that since $f$ has a dense orbit trapped inside some closed rectangle $ABCD\subseteq \mathbb{D}$ the dynamics of $f$ cannot converge to some $\zeta\in S^1$. This implies $f$ has to be conjugate to some rotation. We now claim that it cannot be a rotation either. To see why, note that if it was a rotation, $f$ has to be conjugate to an irrational rotation of $\mathbb{D}$ - as it includes a dense orbit $\{x_n\}_{n\in\mathbb{Z}}$ which is the limit of all the $\Gamma^i_{n(j)}$. Since $\{x_n\}_{n\in\mathbb{Z}}$ is an orbit dense in some Cantor set, this is equally impossible - the reason being that the closure of any infinite orbit for an irrational rotation is a circle, not a Cantor set.

All in all, we have proven that $f$ cannot be a conformal automorphism of the disc as it does not obey the conclusions of the Wolf-Denjoy theorem. It follows there exists some $j_0$ s.t. for all $j>j_0$ the map $f_j$ is not conformal. This proves that for all sufficiently large $j$, the map $f_j$ is $K$--Quasiconformal for some $K>1$ - hence its dilatation in $\mathbb{D}$, $\mu_j$, is not identically zero. Consequently,  $Dil(\Gamma^i_{n(j)})\ne\{0\}$ hence $Dil(C)\ne\{0\}$.

We now prove the second assertion of the theorem. To do so, consider two isotopies $f_t,f'_t:ABCD\to\mathbb{R}^2$, $t\in[0,1]$ defining routes to chaos $C$ and $C'$ s.t. the following holds:
        
        \begin{enumerate}
            \item  $f_1$ and $f'_1$ are deformed horseshoes.
            \item $C$ and $C'$ include different braids - i.e., $f_1$ and $f'_1$ cannot be deformed on their invariant sets to the same Smale horseshoe (see Figure \ref{diffe}).
        \end{enumerate}

We now prove that under these assumptions we have $E(C)\ne E(C')$ - by Definition \ref{chaoticdil} this would suffice to show that the Conformal Indices of $C$ and $C'$ differ. To this end, write $C=\{C_i\}_{i\in\mathbb{N}}$, $C'=\{C'_j\}_{j\in\mathbb{N}}$ and let $\gamma$ be a braid on some cascade $C_i\subseteq C$ s.t. $\gamma\not\in C'$. As $\gamma\not\in C'$, by definition we know that given any braid $\delta$ in the cascade $C'_i$, $\delta$ does not force the existence of $\gamma$. This implies that as we straighten the dynamics of $f'_1:ABCD\to\mathbb{R}^2$ to a horseshoe map $H':ABCD\to\mathbb{R}^2$ by the Betsvina-Handel algorithm, $\gamma$ is not realized as a periodic orbit for $H'$.

Now, let $f_\delta:\mathbb{D}\to\mathbb{D}$ be some dynamically minimal map in $E(C'_j)$ (for some $j$) which realizes $\delta$ as a periodic orbit (possibly after gluing it with some other braids per Theorem \ref{betshan}). As all the periodic orbits of $f_\delta$ are, by minimality, also braid types appearing in $H':ABCD\to \mathbb{D}$ we conclude that it cannot generate $\gamma$ as a periodic orbit. Letting $f_\gamma:\mathbb{D}\to\mathbb{D}$ be any dynamically minimal map in $E(C_i)$ which generates $\gamma$ as a braid, we conclude $f_\gamma\not\in E(C'_j)$. As $\delta$ (and hence $j$) was chosen arbitrarily, we conclude $f_\gamma\not\in E(C'_j)$. By $E(C')=\cup_j E(C'_j)$, $f_\gamma\in E(C)$ we conclude $E(C)\ne E(C'_j)$.
\end{proof}

Before concluding this subsection, we give two examples for how the ideas presented above (and in particular, in Theorem \ref{ch} and Theorem \ref{PA}) can be applied to study concrete examples of dynamical systems. 

\subsubsection{Routes to chaos in the Henon map}
\label{henap}
Given two parameters $a\in\mathbb{R}$ and $b>0$, define the Henon map by $H_{a,b}(x,y)=(a-x^2-by,x)$. This diffeomorphism of $\mathbb{R}^2$, originally introduced in $[53]$, was proven to generate complex dynamics in some open set of parameters (for more details, see $[56, 58]$). In this subsection we give an example of how our results can be applied to the Henon map. To this end, we first recall several facts:

\begin{itemize}
    \item For every $b>0$, $H_{a,b}$ is orientation preserving.
    \item There exist an open parameter range $a,b$ where the Henon map generates complex dynamics - for the details, see $[56, 58]$.
    \item There is an open parameter range $O$ in the $(a,b)-$plane at which there is a topological rectangle $ABCD$ s.t. $H_{a,b}:ABCD\to\mathbb{R}^2$ is a Smale horseshoe map $[55]$. We refer to these parameters as the \textbf{Devaney-Nitecki} \textbf{parameters} and denote them by $\mathbb{DN}$. 
    \item As we vary the parameters $(a,b)$ towards $\mathbb{DN}$, the horseshoe is formed by a route to chaos composed of period-doubling cascades (see Example 1 in Section 2 of $[27]$).

\end{itemize}

We now apply these facts to prove the next result, which is an easy corollary of Theorem \ref{PA}:

\begin{Corollary}
\label{henoncor}    Consider the Henon map $H_{a,b}(x,y)=(a-x^2-by,x)$, where $b>0$. Then, there exists an open set of parameters in the $(a,b)$--plane, bordered by $\overline{\mathbb{DN}}$, in which $H_{a,b}$ generates infinitely many periodic orbits.
\end{Corollary}
\begin{proof}
Let $f_t:ABCD\to\mathbb{R}^2$, $t\in[0,1]$ be an isotopy of the Henon map s.t. $f_1\in \overline{\mathbb{DN}}, f_0\not\in \overline{\mathbb{DN}}$. Since $f_1:ABCD\to\mathbb{R}^2$ is a deformed Smale horseshoe by $[55]$, from Example 1 in $[27]$ we can choose the isotopy s.t. it defines a period-doubling route to chaos, $C$. By Theorem \ref{PA} we know that the Conformal Index of $C$ is non-zero, and by Theorem \ref{ch} this implies there exists some braid $\gamma$, realized as a finite collection of periodic for $f_1$, s.t. $\gamma$ is not finite order. Therefore, by Theorem \ref{ch} we conclude that there exists some $0\leq t_0<1$ s.t. for all $t>t_0$, $f_t$ also generates the same braid and also generates all the periodic orbits forced by $\gamma$.

Finally, we recall that as shown in Example 1 of $[27]$, the collection of all curves in the $(a,b)$--plane defining an isotopy as above includes an open set. Or, in other words, there exists a rectangle $R$ in the $(a,b)$--parameter plane, with one side on $\partial\mathbb{DN}$, foliated by curves defining isotopies as described above. Let $\{\gamma(t)_\rho\}_{\rho\in\Delta}$, $t\in[0,1]$  denote the said curves indexed by $\rho$, parameterized s.t. $\gamma_\rho(1)\in\partial R\cap\partial\mathbb{DN}$. Moreover, for each $\rho\in\Delta$ set $t^\rho_0$ to be the minimal $t^\rho_0$ s.t. for all $t>t^\rho_0$ the Henon map corresponding to $\gamma_\rho(t)$ has infinitely many periodic orbits.

We now claim that $1\not\in\overline{\{t^\rho_0\}_{\rho\in\Delta}}$. It is easy to see that if this is the case, Corollary \ref{henoncor} would follow. We do so by contradiction, i.e., assume $1\in\overline{\{t^\rho_0\}_{\rho\in\Delta}}$. This implies that there exists some isotopy of the Henon map $f_t:ABCD\to\mathbb{R}^2$ defined by a curve in $\overline{R}$, generating a route to chaos with an endpoint at $\overline{\mathbb{DN}}$, and moreover, for all $t<1$, $f_t$ has only finitely many periodic orbits in $ABCD$. As $f_1:ABCD\to\mathbb{R}^2$ is a deformed horseshoe, by Theorem \ref{PA} we know that the Conformal Index of its route to chaos is non-zero, i.e., the isotopy satisfies the conclusion of Theorem \ref{ch}. But since we assume that $f_t:ABCD\to\mathbb{R}^2$ has only finitely many periodic orbits in $ABCD$ for all $t\in[0,1)$, this is a contradiction.
\end{proof}

Before moving on, we remark that similar results for the Henon map were obtained in $[57]$ using pruning theory for two-dimensional dynamics (see $[54]$ for more background).

\subsubsection{Obtaining geometric models for homoclinic perturbations}
\label{shill}
We now apply Theorem \ref{ch} and Theorem \ref{PA} to study the Shilnikov scenario. To begin, without any loss of generality, assume  that $0$ is a fixed point for some sufficiently smooth vector field given by the following formulas: 

\begin{equation} \label{shil}
\begin{cases}
\dot{x} = -\rho x-\omega y+F_1(x,y,z), \\
 \dot{y} = \omega x-\rho y+F_2(x,y,z),\\
 \dot{z}=-\gamma z+F_3(x,y,z),
\end{cases}
\end{equation}
where $F_1,F_2,F_3,\rho,\omega$ and $\gamma$ satisfy the following (see the illustration in Figure \ref{homoclinic}):

\begin{itemize}
    \item $F_i$, $i=1,2,3$, are smooth and vanish at the origin (together with their first partials).
    \item $\rho>0,\omega\ne0$ and $\gamma>0$, i.e., the origin is a saddle-focus with a two-dimensional stable manifold and a one-dimensional unstable manifold.
    \item The eigenvalues $\gamma$ and $\rho\pm i\omega$ of the Jacobian matrix at the origin satisfy the resonance condition $\frac{\rho}{\gamma}<1$.
    \item There exists a homoclinic loop $\Gamma$ to the origin.
\end{itemize}

\begin{figure}[h]
\centering
\begin{overpic}[width=0.45\textwidth]{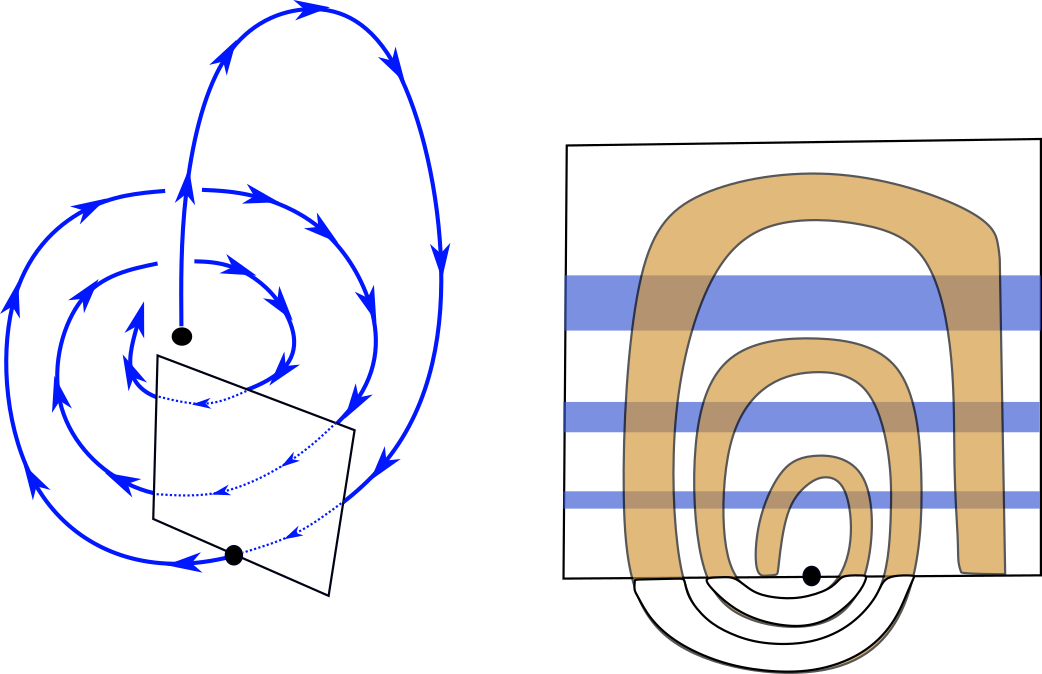}

\end{overpic}
\caption{\textit{The Shilnikov scenario. The existence of a homoclinic trajectory $\Gamma$ with the prescribed resonance conditions cause the first-return map from the cross-section $\Sigma$ (the rectangle), transverse to $\Gamma$, to create infinitely many Smale horseshoes, accumulating on $\Gamma\cap\Sigma$ (as sketched on the right)}}\label{homoclinic}

\end{figure}

As proven in $[60]$, under these assumptions there exists a cross-section $\Sigma$, transverse to $\Gamma$, at a single point $p_0$, and a collection of rectangles $\{R_n\}_{n\in\mathbb{N}}$ accumulating on $p_0$, s.t. for all $n$ the first-return map $f:R_n\to\Sigma$ is a Smale horseshoe. We refer to this fact as \textbf{Shilnikov's theorem}. In fact, one can say more - namely, given any smooth curve of vector fields $F_t:\mathbb{R}^3\to\mathbb{R}^3$, $t\in[0,1]$ s.t. $F_1$ satisfies the assumptions (and conclusions) of Shilnikov's theorem, for any sufficiently large $t\in[0,1)$ the corresponding first-return map $f_t:R_n\to\Sigma$ for $F_t$ (where $f=f_1$) will be a Smale horseshoe map, which can be deformed on its invariant set to the Smale horseshoe map as in Figure \ref{homoclinic} (see $[60]$ or Theorem 13.8 in $[61]$). We immediately derive:

\begin{Corollary}
\label{shilPA}   For any $k>0$ there exists some $t_k$ s.t. for all $r>t_k$, the vector field $F_r$ suspends at least $k$ invariant sets, $I_1,...,I_k$ s.t. the dynamics of the first-return map on each one can be factored to some subshift of finite type. 
\end{Corollary}
\begin{proof}
Given any $n>0$, by the proof of Shilnikov's theorem, we know that the first return map $f:R_n\to \Sigma$ for $F_1$ is a $\cap$-- Smale horseshoe. This implies that for every such rectangles there exists some finite collection of periodic orbits for the flow of $F_1$ which intersect $R_n$ in a non-finite order braid, $\gamma$. Now, consider some perturbation of $F_1$ to a vector field $F_t$, which isotopes $f:R_n\to\Sigma$ to another map $f_t:R_n\to\Sigma$. Using a similar argument to the one used in the proof of Theorem \ref{ch}, it is easy to see all the periodic orbits forced by $\gamma$ in $R_n$ w.r.t. $f=f_1$ persist as we isotope it to $f_t$. As $f_t$ is a first-return map for a smooth flow, this proves that for all $t$ sufficiently close to $1$, $F_t$ suspends all the dynamics forced by the braid $\gamma$. Let $I$ denote the corresponding invariant set - it follows that for all sufficiently large $t$, the first-return map $f_t:R_n\to\Sigma$ is continuous around $I\cap R_n=J$.

By Theorem \ref{betshan} we know that the dynamics of $f_t$ on $J$ can be factored to some subshift of finite type, $\Sigma_1\subseteq\{1,2\}^\mathbb{N}$, determined canonically by the braid $\gamma$. Now, given any $k>0$, choose some $k$ rectangles, $R_1,...,R_k\subseteq\Sigma$ - applying these arguments simultaneously to each one of these rectangles, it is easy to see that each one defines an invariant set suspending some pseudo-Anosov or reducible braid dynamics. Denoting these sets by $I_1,...,I_k$, we conclude that for all sufficiently small perturbation of $F_1$ to another vector field $F_t$, provided $t\in(0,1)$ is sufficiently close to $1$, it will also suspend all of these sets and their symbolic dynamics.
\end{proof}

The Shilnikov scenario is well-known to be connected with the onset of chaos for flows, and it is often connected with the appearance of complex dynamics in attractors (see, for example, $[62, 63]$ and the references therein). We further remark the existence of a route to chaos per Definition \ref{route} as $t\to1$ was investigated in $[64]$ (and, indeed, Corollary \ref{shilPA} can be thought of as a slight extension to the results of $[64]$). In particular, Corollary \ref{shilPA} should be interpreted as a result about some lower bound for the dynamical complexity in the perturbation of the Shilnikov scenario.

\subsection{Numerical invariants}
\label{numerin}

The conformal invariant can be difficult to apply directly to the study of concrete examples of routes to chaos, mostly due to its theoretical nature. Therefore, in this subsection we define numerical invariants. Much like the Conformal Index introduced in the previous subsection, the new invariants are also based on the braid-theoretical analysis of the cascade. By their definition, these new invariants will also persist under relative isotopies of the route to chaos (as defined in Definition \ref{type}). However, unlike in the previous subsection, the invariants will be constructed using algebraic tools (the reader is encouraged to consult with the results stated in the algebraic part of Prerequisites  - i.e., Subsection \ref{braidef}, Subsection \ref{indexdef} and Subsection \ref{numdef}). To give a brief description, in this section we study bifurcation structures using group-theoretic methods and the Burau representation to move from braids to numbers. These ideas can be thought of as analogues of Theorem \ref{bers} which  was used to translate the braids appearing along the cascade into complex structures. We will come back to this analogy at the end.\\

To begin, recall that a given route to chaos $C$ is defined by an isotopy $f_t:ABCD\to\mathbb{R}^2$, $t\in[0,1]$ (where $ABCD$ is some topological rectangle) and a collection of period-doubling cascades $\{C_i\}_{i \in \mathbb{N}}$ s.t. $C_i=\{\gamma_i,t^i_n\}_{n\in\mathbb{N}}$. In detail, $\gamma_i$ is the initial braid and $t^i_n$, $n\geq1$ corresponds to the period-doubling points along the cascade where $\gamma_i$ is transformed to the braid $\gamma^n_i$. In particular,  we used this formalism to define a collection of sequences of braid subgroups $V_i=\{\beta^n_i\}_{n\in\mathbb{N}}$ corresponding to the state of $f_t$ w.r.t. the cascade - i.e., $\beta^i_n$ is the braid subgroup defined by the gluing of $\gamma_i,\gamma^1_i,...,\gamma^n_i$ by the Betsvina-Handel algorithm (see the discussion immediately before Theorem \ref{assol}). With these ideas in mind, we first prove the following:

\begin{Proposition}
  \label{num1}  Let $C_i$ be a period-doubling cascade, for some $i\in\mathbb{N}$. Then, for any natural $N \ge 2$ we can associate with the cascade $C_i$ an irrational number $N(C_i)$ (the \textbf{$N$-Index-Invariant}$_i$) and a p-adic integer $N_p(C_i)$ (the $p-$\textbf{$N$-Index-Invariant}$_i$) for any prime $p$, and for any non-zero complex number $t$ -- the sequence of complex numbers $Tr(C_i)$ (the \textbf{$t$-Trace-Invariant}$_i$). 
  \end{Proposition}
\begin{proof} Let $\beta^n_i$ be a subgroup of the braid group on $k$ strands $B_k$.

Since $\beta^n_i$ is (almost) never a finite-index subgroup of the corresponding braid group (as it is generated by only one braid), we have to fix a natural $N \ge 2$ and to consider the symplectic representation of the braid group in $SL(k,\mathbb{Z}_N)$, the Special Linear group over the ring of integers modulo $N$ (we use the Strong Approximation theorem here, see Theorem \ref{strongapp}). Since $SL(k, \mathbb{Z}_N)$ is a finite group, all its subgroups are finite-index subgroups. Then we define a natural number $c^n_i$ as the index of the symplectic representation in $SL(k,\mathbb{Z}_N)$ of $\beta^n_i$ relatively the symplectic representation of the braid group $B_k$ in $SL(k,\mathbb{Z}_N)$ (the Burau representation doesn’t mean to be surjective).

This allows to define the continued fraction (it is irrational due to $[22]$, see Subsection \ref{numdef}):
    \begin{equation*}
       N(C_i)= c^1_i+\cfrac{1}{c^2_i+\cfrac{1}{c^3_i+\cfrac{1}{c^4_i+\cfrac{1}{c^5_i+ \cdots\vphantom{\cfrac{1}{1}} }}}}
    \end{equation*}
and the corresponding p-adic integer:
\begin{equation*}
       N_p(C_i)= \sum\limits_{n=1}^\infty (c^n_i \mod p) p^{n-1}
    \end{equation*}

Given any non-zero $t\in\mathbb{C}$, we define $Tr(\Gamma^i_n)$ as the traces of the Burau representation for $\Gamma^i_n, n \in \mathbb{N}$ (not necessarily with $t= -1$, i.e., not necessarily the symplectic representation -  see Subsection \ref{indexdef}). This yields $Tr(C_i)=\{Tr(\Gamma^i_n)\}_{n\in\mathbb{N}}$. 
\end{proof}

We can visualize the p-adic integers (i.e., the $p-N$-Index-Invariants) as a subset of the plane using a fractal-like construction from $[26]$.  Moreover, there are pretty simple homeomorphisms between the ring of 2-adic integers and the Cantor set, and between the ring of 3-adic integers and the Sierpinski triangle.

\begin{Proposition}
  \label{num2}  Let $C=\{C_i\}_{i\in\mathbb{N}}$ be a route to chaos. Then, for any natural $N \ge 2$ we can associate with $C$ an irrational number $N(C)$ (the \textbf{$N$-Index-Invariant}) and for any non-zero complex number $t$ -- an enumeration of complex numbers $Tr(C)$ (the \textbf{$t$-Trace-Invariant}). 
  \end{Proposition}
\begin{proof}
Choose some $N\geq2$. Due to the braid ordering (see Subsection  \ref{braidef}), we have the linear order for $\cup_i V_i=\{\beta^n_i\}_{i, n\in\mathbb{N}}$ coming from $\{\Gamma^i_n\}_{i, n \in \mathbb{N}}$. In more detail, recall the braid $\Gamma^i_n$ from Theorem \ref{assol}, which generates the minimal braid group $\beta^n_i$, i.e., $\beta^n_i$ is the minimal braid group which include $\Gamma^i_n$. Thus, by ordering the set $\{\beta^n_i\}_{i, n \in \mathbb{N}}$ via the braid ordering on $\{\Gamma^i_n\}_{i, n \in \mathbb{N}}$, we can associate a number $c^i_n$ with every $\beta^i_n$ and derive a continued fraction in the same manner as in Proposition \ref{num1}. Similarly, the same argument implies that for all non--zero complex number $t$ we can itemize the Burau representation for $\{\Gamma^i_n\}_{i, n \in \mathbb{N}}$.
\end{proof}

At this point, one is led to ask the following question: given $N\geq2$, what is the relation between the $N$-Index-Invariant $N(C)$ and the $N$-Index-Invariants$_i$, i.e., the collection $\{N(C_i)\}_{i\in\mathbb{N}}$?  Unfortunately, it is not possible to relate $N$-Index-Invariant with $N$-Index-Invariants$_i$, as there is no linear forcing order on the collection of braids as in the Sharkovskii theorem (for example, it is meaningless to compare two finite order braids; see also Theorm \ref{assol}).\\

We further remark it is easy to see the similarity in the definitions between these numerical invariants and the Conformal Index introduced in the previous subsection - in the sense that all these invariants form "translations" of the braids defined by the cascade into different objects which can be studied more easily. However, unlike the Conformal Index, this time the invariants \textbf{can be approximated numerically}, and later we will describe a possible theoretical algorithm explaining explicitly how to do so (see Subsection \ref{numerin2}). However, before doing so, we first prove that our choice of name is justified - i.e., that the $N$-Index- and $t$-Trace- Invariants truly form invariants of a route to chaos. Namely, we prove:

\begin{Theorem}
 \label{numinv} For any $N\geq2$ and $t\ne0$, the $N$-Index-Invariant and the $t$-Trace-Invariant are all topological invariants of a route to chaos, $C=\{C_i\}_{i\in\mathbb{N}}$ - i.e., they are preserved under isotopies of relative to $C$, hence they form invariants of the bifurcation class of the period-doubling route to chaos.
\end{Theorem}
\begin{proof} To begin, recall that per Theorem \ref{assol} and Definition \ref{type}, for a given route to chaos we can always associate a sequence of braid subgroups $V_i=\{\beta^n_i\}_{n\in\mathbb{N}}$. In particular, the collection $V =\{V_i\}_{i \in \mathbb{N}}$ is already a topological invariant, where each $V_i$ is associated with $C_i$. Notice again the braid group $\beta^n_i$ is always generated via the single braid $\Gamma^i_n$, which is the canonical gluing of $\gamma_i,...,\gamma^n_i$ per the Betsvina-Handel algorithm (it is a cyclic group - a group where every element can be expressed as a power of that single element (or its inverse)). Since relative deformations do not change the braid type of $\Gamma^i_n$, the statement becomes obvious.
\end{proof}

Let us mention here two related concepts from representation theory.  The first is Schur’s lemma, which says that there are no non-zero homomorphisms between distinct (i.e., non-isomorphic) irreducible representations and any non-zero morphism among isomorphic irreducibles is an isomorphism. The second is the character of a group representation, which is a function on the group that associates to each group element the trace of the corresponding matrix - the character carries the essential information about the representation, since, for example, a complex representation of a finite group is determined (up to isomorphism) by its character (isomorphic representations always have the same characters) and if a representation is the direct sum of subrepresentations, then the corresponding character is the sum of the characters of those subrepresentations.\\

As an immediate corollary of Theorem \ref{numinv}, we have the following:

\begin{Corollary}
    \label{sequence} The sequence of $N$-Index-Invariants for $N \ge 2$, the \textbf{Index-Invariant}, and the set of $t$-Trace-Invariants for $t \in \mathbb{C}$, the \textbf{Trace-Invariant}, are topological invariants of a route to chaos.
\end{Corollary}
   
Having proven that the invariants are invariants of the relative isotopy, we are led to the following important question - how sensitive are these invariants? At present, we do not have a complete answer to this question, with the main difficulty stemming from the fact that the Burau representation is \textbf{not} in general faithful (injective) - and moreover, in general scenarios is kernel is unknown. However, we do have several conjectures on how this question can be settled. Before presenting them (and regardless of the answer), we first state another corollary of Theorem \ref{numinv}:

\begin{Corollary}
\label{sense}
    Assume the isotopies $f_t:ABCD\to\mathbb{R}^2$ and $g_t:ABCD\to\mathbb{R}^2$, $t\in[0,1]$ define routes to chaos $C$ and $C'$ s.t. there exists some $N\geq2$ (some non-zero complex number $t$) for which the $N$-Index-Invariants  ($t$-Trace-Invariants) are different. Then, $f_t:ABCD\to\mathbb{R}^2$ \textbf{is not} isotopic to $g_t:ABCD\to\mathbb{R}^2$ relative to $C$.
\end{Corollary}

We point out here that trace and dilatation are terms used in physics to describe changes in volume and energy (note that trace is related to the derivative of the determinant -- Jacobi's formula). This leads us to suspect a connection between the Conformal Index (the end point) and Trace-Invariant (the paths). In addition, we also believe Corollary \ref{sense} can be reversed -- at least for the Trace-Invariant.\\

To precisely state our conjecture, we first recall the \textbf{Frobenius reciprocity} $[89]$: the induction functor for representations of groups is left adjoint to the restriction functor. In other words, there exists an indirect connection between indexes and traces, since the Frobenius reciprocity connects induced and restricted representations of a group and its subgroup. In detail, the character of the given group can be expressed in terms of the characters of the given subgroup, the character of the induced representation is provided by multiplying the index of the subgroup by the character of the representation of the subgroup. This motivates us to conjecture the following:

\begin{Conjecture}
\label{ind-con}
The Index-Invariant is a complete set of invariants for a route to chaos. 
\end{Conjecture}

Since we believe a proof of Conjecture \ref{ind-con} has to be related to the properties of the kernel of the Burau representation, we now discuss what has been known about the kernel. First, in his influential paper $[11]$, Squier briefly considered the specializations of the Burau representation at roots of unity. In addition, the hyperelliptic Torelli group can be identified with the kernel of the Burau representation evaluated at $t=-1$, see $[81]$ (recall that the said group is generated by Dehn twists about separating curves that are preserved by the hyperelliptic involution). Finally, the basic fact about the kernel of the Burau representation is that it is a subset of the commutator subgroup of the pure braid group $[8, 14]$.\\

At this point we remark that if true, Conjecture \ref{ind-con} implies that even though the Index-Invariant is blind to the order of appearance of periodic orbits, it does "see" all the other topological data. Moreover, it infers the following conjecture:

\begin{Conjecture}
\label{ordercon} Assume $C$ and $C'$ are different routes to chaos. Then, they have the same Conformal Index precisely when $C$ and $C'$ define the same braid types. 
\end{Conjecture}
 
If true, Conjecture \ref{ordercon} indicates that there exists at most one collection of braids in a route to chaos which defines a zero Conformal Index (i.e., where each braid on every cascade corresponds to a Möbius transformation per Theorem \ref{bers}). In addition, it would also say that intuitively we should always expect a route to chaos to include a pseudo-Anosov braid - which would implicate that in general, Theorem \ref{ch} can be generalized to most (if not all) route to chaos. At present, we do not know how to solve this conjecture - however, we will indicate a possible way to do so.\\

Inspired by Teichmüller theory, we believe the solution of Conjecture \ref{ordercon} should follow by finding a metric structure on the space of all routes to chaos. To do so, recall that for a given route to chaos $C$, we define the \textbf{boundary of $C$}, $[C]$, as the collection of limit dilatations of sequences $\{\mu_{j_k}\}_k$, $\mu_{j_k}\in Dil(C_{j_k})$ (where $Dil(C_i),i\in\mathbb{N}$ are as in Definition \ref{chaoticdil}). Finally, define $T_{chaos}$ to be the collection of all $[C]$, for all routes to chaos. We now define a map $d:T^2_{chaos}\to\mathbb{R}^+$ by the formula:
\begin{equation*}
    d([C],[C'])=\inf\{\ln(K(f^{-1}\circ g)): f\in [C],g\in [C']\}
\end{equation*}
Using the properties of Q.C. maps, it immediately follows that $d$ is symmetric and satisfies the triangle inequality - hence, it is a pseudometric and we conjecture the following:
\begin{Conjecture}
\label{metric} $d$ is a metric on $T_{chaos}$. 
\end{Conjecture}

In other words, combined with forcing phenomena, this should tell that the possible relative deformations of a given route to chaos are strongly constrained by the order of the braids - and as such, it raises one important question: what is the distribution of finite-order braids inside routes to chaos?

\subsubsection{Computation}
\label{numerin2}

Let $f_t;ABCD\to\mathbb{R}^2$ define some route to chaos $C=\{C_i\}_{i\in\mathbb{N}}$. In this subsection, we indicate how the Index-Invariant can theoretically be approximated and possibly even computed. For that we need to deal with the following three tasks, as the algorithm is simple - by iterating braids, check that it is a part of some cascade and if it is so, calculate the corresponding index and update the current step continued fraction: 
\begin{enumerate}
    \item Given a braid $\gamma$, we need to calculate its corresponding Burau representation. 
    \item Due to the definition of the Index-Invariant (see Proposition \ref{num2} and Corollary \ref{sequence}), we need a precise description of the linear order on braids. This is related to the word and conjugacy problem for braids (see $[79, 80]$). 
    \item Given a braid $\gamma$, we need to verify that it lies on some cascade $C_i$ - and if that is the case, we need to plug it in the approximation formula (see Propositions \ref{num1}, \ref{num2}).\\
\end{enumerate}

We now illustrate how each one of these problems can be studied to derive an algorithm.\\

\begin{figure}[h]
\centering
\begin{overpic}[width=0.25\textwidth]{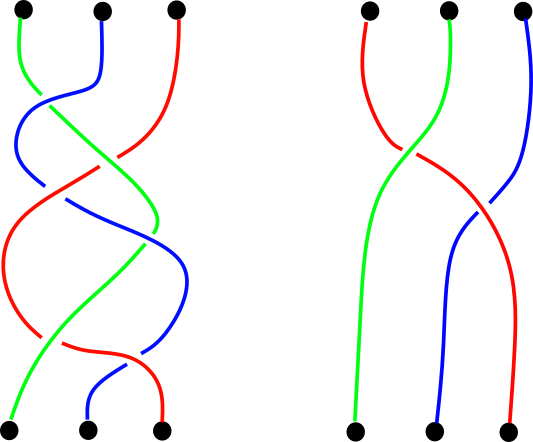}
\end{overpic}
\caption{\textit{Two braids appearing as periodic orbits (or rather, a collection of periodic orbits) in the $2n-$body problem (see Figure 5 in $[65]$).}}\label{2nbody}
\end{figure}

The third task can be studied numerically: for example, period-doubling bifurcations for orbits can be located numerically, see $[74]$. Moreover, given a collection of periodic orbits $\gamma_i,...,\gamma^n_i$, one can glue them together with the Betsvina-Handel algorithm using computer programs as in $[72, 73]$. \\

To deal with the second task, we know that the linear ordering of braids, which was first discovered using results of self-distributive
algebra, has now received several alternative constructions, see $[2, 76, 77, 78]$. As shown in $[77]$ one can explicitly describe this order for $B_3$. Therefore, assuming these ideas can be extended to any $B_n$ (see $[2]$), the second task is solvable too.\\

The first task is composed of direct calculations, so we give an example for how the Burau representation is computed for some "real-life" braids. Consider the right braid from the illustration in Figure \ref{2nbody} - it is expressed as the product of the standard generators of the braid group $B_3$ this way: $\sigma_1^{-1} \sigma_2$. Similarly, the left braid corresponds to $\sigma_1^{-1} \sigma_2 \sigma_1^{-1} \sigma_2 \sigma_1^{-1} \sigma_2$. By definition, the Burau representation for $B_3$ is defined by the following map:

$$\sigma_1 \longrightarrow \begin{pmatrix}
     1-t & t  & 0 \\
     1 & 0 & 0 \\
     0 & 0 & 1
\end{pmatrix}, \sigma_2 \longrightarrow \begin{pmatrix}
    1 & 0 & 0  \\
    0 & 1-t & t \\
    0 & 1 & 0 
\end{pmatrix}.$$

Therefore, by computing these products the Burau representation is found.\\

To summarize this discussion, the Index-Invariant can be approximated through direct computation once all of these tasks have been completed.

\section{Discussions}

Before we conclude this paper, we would like to make some final remarks about how our ideas can probably be expanded and how they possibly relate to other topics. As can be seen from our arguments, we are trying to describe the bifurcations of two-dimensional maps using topological methods. In fact, as stated in the introduction, one could also describe our ideas as an attempt to inject representation-theoretic methods into bifurcation theory. Pushing our ideas further, it appears natural to find applications in bifurcation theory to the following three concepts:\\

i) \textbf{Braided monoidal categories} $[83]$. A braided monoidal category is a category with a tensor product and an isomorphism called ‘braiding’. The braiding isomorphism allows one to switch two objects in the tensor product such that hexagon identities, encoding the compatibility of the braiding with the associator for the tensor product, are satisfied.\\ 

ii) \textbf{Cluster algebras} $[84]$. A cluster algebra is constructed from an initial seed as follows: if we repeatedly mutate the seed in all possible ways, we get a finite or infinite graph of seeds. In more details, the graph is defined like so -- two seeds are joined by an edge if one can be obtained by mutating the other. The underlying algebra of the cluster algebra is the algebra generated by all the clusters of all the seeds in this graph.\\ 

iii) \textbf{Shuffle algebras} $[85]$. A shuffle algebra is an algebra with a basis corresponding to words on some set, whose product is given by the shuffle product $X \omega Y$ of two words $X, Y$. In particular, one could think of such algebras as the sum of all ways of interlacing these words.\\

The reason we chose these concepts is because they appear similar to the notion of a bifurcation class (see Corollary \ref{bifurcationclass}). In detail, recall that given an isotopy $f_t:ABCD\to\mathbb{R}^2$, $t\in[0,1]$ we define by $Per\subseteq ABCD\times[0,1]$ the set $\{(x,t)|\exists n>0, f^n_t(x)=x\}$, and assume the isotopy defines a route to chaos $C$. Let us denote by $P\subseteq Per$ the collection of periodic points corresponding to period-doubling cascades - the bifurcation class of $C$ was defined as the collection of isotopies $g_t:ABCD\to\mathbb{R}^2$, which can be deformed to $f_t:ABCD\to\mathbb{R}^2$ relative to $C$ in $ABCD\times[0,1]\setminus P$. Due to the fact that the set $P$ can be visualized as a "branched braid" (see Figure \ref{branched}) on infinitely many strands, it appears natural that the algebraic concepts mentioned above could help studying complex objects as bifurcation structures.\\

In addition, it appears our ideas open the possibility of using a variety of new tools in topological dynamics. Even though they are not directly related to our results, we list them below as we believe they justify further research.

\subsection{Normal numbers} 
\label{normal}
A real number is said to be simply normal in an integer base $b$ if its infinite sequence of digits is distributed uniformly in the the sense that each of the $b$ digit values has the same density $1/b$. A number is said to be normal in base $b$ if, for every natural $k$, all possible strings $k$ digits long have density $b^{-k}$. A number is said to be \textbf{normal} if it is normal in all integer bases greater than or equal to $2$.\\

The collection of all normal sequences are closed under finite variations: adding, removing, or changing a finite number of digits in any normal sequence leaves it normal. Similarly, if a finite number of digits is added to, removed from, or changed in any simply normal sequence, the new sequence is still simply normal.\\

Almost all real numbers are normal (meaning that the set of non-normal numbers has Lebesgue measure zero - i.e. the Borel-Cantelli lemma). Hence, the proof is not constructive, and only a few specific numbers have been shown to be normal. Moreover, if $f(x)$ is any non-constant polynomial with real coefficients such that $f(x) > 0$ for all $x > 0$, then the real number $0.[f(1)][f(2)][f(3)]\dots$, where $[f(n)]$ is the integer part of $f(n)$ expressed in base $b$, is normal in base $b$ $[21]$.\\

We know that N-Index-Invariants for $N \ge 2$ are irrational numbers and thus are probably normal numbers.

\subsection{The volume conjecture}
\label{volume}
The volume conjecture states that for a hyperbolic knot $K$ in the three-sphere $S^3$ the asymptotic growth of \textbf{the colored Jones polynomial} of $K$ is governed by the hyperbolic volume of the knot complement $S^3 \setminus K$ (see $[41, 42]$). The conjecture relates two topological invariants, one combinatorial and one geometric.\\

Jones' original formulation of his polynomial came from study of operator algebras. It resulted from a kind of "trace" of a particular braid representation into an algebra (the Temperley–Lieb algebra), which originally arose while studying certain models, e.g. the Potts model, in Statistical Mechanics.\\

The $N$-colored Jones polynomial is the \textbf{Reshetikhin–Turaev invariant} $[43]$ associated with the $(N+1)$-irreducible representation of the quantum group $U_{q}(\mathfrak{su}_2)$ (in this scheme, the Jones polynomial is the 1-colored Jones polynomial). In other words, it is intimately related to representation theory of the special unitary group $SU(2)$ of $2 \times 2$ matrices. The $N$-colored Jones polynomial can be given a combinatorial description using the $N$-cabling $[44]$. The volume conjecture for $SU(n)$ invariants is given in $[45]$ (a special thing about $SU(n)$ is that it is - up to isomorphism - the only compact and connected Lie group such that the complexification of its Lie algebra is isomorphic to $\mathfrak{sl}(n, \mathbb{C})$).\\

At this point we recall that when one moves from two-dimensional homeomorphisms to three-dimensional flows, braids are replaced with knots. Therefore, we believe the invariants described above could be applied to extend our results from a two-dimensional context into that of three-dimensional flows.

\subsection{Higher complex structures}
\label{highcomplex}

In $[48, 49]$ a new geometric structure on surface s was introduced - higher complex structure, whose moduli space is conjecturally diffeomorphic to \textbf{Hitchin’s component}. This is an approach to higher Teichmüller theory $[50]$, which can be briefly described as follows: for a Riemann surface $S$, Hitchin's component is a connected component of the character variety
$Rep(\pi_1(S), PSL(n, \mathbb{R})) = Hom(\pi_1(S), PSL(n, \mathbb{R}))/ PSL(n, \mathbb{R})$,
parametrized by holomorphic differentials.\\

We recall that a complex structure on a Riemann surface is characterized by a Beltrami differential (i.e., a complex dilatation) $\mu \in \Gamma(K^{-1} \otimes \overline{K})$, where $K$ is the canonical bundle. The Beltrami differential determines the notion of a local holomorphic function $f$ by the condition $(\overline{\partial} - \mu \partial)f = 0$ (see, for example, $[31]$).\\

We further recall that the Beltrami differential determines a linear direction in the tangent bundle - namely, the direction generated by the vector $\overline{\partial} - \mu \partial$. The idea of higher complex structures is to replace the linear direction by a polynomial direction, that is, by an n-jet of a curve at the origin (for more details, see $[48, 49]$; in addition, see also $[51]$ about Anosov representations).\\

As the tools used in Teichmüller theory are strictly two-dimensional, they cannot be applied to homeomorphisms of $\mathbb{R}^n$, $n\geq3$. However, assuming Theorem \ref{bers} can be somehow generalized to higher dimensions using higher complex structures, this could hopefully be achieved. However, such results probably require finding some topological invariant that forces complex dynamics to appear similarly to how braids function in dimension $2$.


~\\
Shanghai Institute for Mathematics and Interdisciplinary Sciences 
\end{document}